\documentclass[12pt,twoside]{amsart}
\usepackage{latexsym,amssymb,amsmath,hyperref}
\usepackage{fullpage}

\usepackage{tikz}
\usetikzlibrary{matrix}
\usetikzlibrary{arrows,patterns}
\tikzstyle{ghostfill} = [fill=white]
         \tikzstyle{ghostdraw} = [draw=black!50]
\usetikzlibrary{decorations.markings}
\tikzstyle arrowstyle=[scale=1]
\tikzstyle directed=[postaction={decorate,decoration={markings, mark=at position .65 with {\arrow[arrowstyle]{stealth}}}}]
\tikzstyle reverse directed=[postaction={decorate,decoration={markings, mark=at position .65 with {\arrowreversed[arrowstyle]{stealth};}}}]

\numberwithin{equation}{section}

\newtheorem{theorem}{Theorem}[section]
\newtheorem{lemma}[theorem]{Lemma}

\newtheorem{corollary}[theorem]{Corollary}
\newtheorem{conjecture}[theorem]{Conjecture}

\theoremstyle{definition}
\newtheorem{definition}[theorem]{Definition}
\newtheorem{def-prop}[theorem]{Definition-Proposition}

\newtheorem{remark}[theorem]{Remark}
\newtheorem{example}[theorem]{Example}

\newtheorem*{acknowledgement}{Acknowledgement}
\newtheorem{notation}[theorem]{Notation}

\DeclareMathOperator{\pol}{pol}

\DeclareMathOperator{\depth}{depth}

\newcommand{\ZZ}{{\mathbb Z}}
\newcommand{\NN}{{\mathbb N}}

\newcommand{\FF}{{\mathbb F}}

\def\B{{\mathcal B}}

\def\D{{\mathcal D}}

\def\F{{\mathcal F}}

\def\X{{\mathcal X}}
\def\H{{\mathcal H}}

\newcommand\bsa{{\boldsymbol a}}

\newcommand\bsx{{\boldsymbol x}}

\newcommand{\cupdot}{\mathbin{\mathaccent\cdot\cup}}

\def\w{\omega}

\def\1{{\bf 1}}
\def\0{{\bf 0}}

\begin{document}


\title[]{Edge ideals of oriented graphs}
\author[H\`{a}]{Huy T\`{a}i H\`{a}}
\address{Department of Mathematics\\
Tulane University\\
New Orleans, LA 70118 }
\email{tha@tulane.edu}

\author[Lin]{Kuei-Nuan Lin}
\address{Academic Affairs\\ 
Penn State Greater Allegheny\\ 
 McKeesport, PA 15132}
\email{kul20@psu.edu}

\author[Morey]{Susan Morey}
\address{Department of Mathematics\\
Texas State University\\
San Marcos, TX 78666
}
\email{morey@txstate.edu}

\author[Reyes]{Enrique Reyes}
\address{
Departamento de
Matem\'aticas\\
Centro de Investigaci\'on y de Estudios
Avanzados del
IPN\\
Apartado Postal
14--740 \\
07000 Mexico City, D.F.
}
\email{ereyes@math.cinvestav.mx}

\author[Villarreal]{Rafael H. Villarreal}
\address{
Departamento de
Matem\'aticas\\
Centro de Investigaci\'on y de Estudios
Avanzados del
IPN\\
Apartado Postal
14--740 \\
07000 Mexico City, D.F.
}
\email{vila@math.cinvestav.mx}

\keywords{Cohen--Macaulay, weighted digraph, edge ideal.}
\subjclass[2010]{Primary 13F20; Secondary 05C22, 05E40, 13H10.}

\begin{abstract}
Let $\D$ be a weighted oriented graph and let $I(\D)$ be its edge ideal. Under a natural condition that the underlying (undirected) graph of $\D$ contains a perfect matching consisting of leaves, we provide several equivalent conditions for the Cohen-Macaulayness of $I(\D)$. We also completely characterize the Cohen-Macaulayness of $I(\D)$ when the underlying graph of $\D$ is a bipartite graph. When $I(\D)$ fails to be Cohen-Macaulay, we give an instance where $I(\D)$ is shown to be sequentially Cohen-Macaulay.
\end{abstract}

\maketitle


\section{Introduction}\label{intro-section}

An \emph{oriented graph} $\D = (V(\D), E(\D))$ consists of an \emph{underlying} simple graph $G$ on which each edge is given an orientation (i.e., a directed graph without multiple edges nor loops). The elements of $E(\D)$ will be denoted by ordered pairs to reflect the orientation, where $(x,y)$ represents an edge directed from $x$ to $y$. An oriented graph $\D$ is called \emph{vertex-weighted} (or simply, \emph{weighted}) if it is equipped with a \emph{weight function} $\w: V(\D) \longrightarrow \NN$.

Let $\D$ be a weighted oriented graph over the vertex set $V(\D) =
\{x_1, \dots, x_n\}$ and let $R = K[x_1, \dots, x_n]$ be a polynomial
ring over a field $K$. For simplicity, let $\w_j = \w(x_j)$ for $j =
1, \dots, n$. The
\emph{edge ideal} of $\D$ is defined to be
$$I(\D) = \big(x_ix_j^{\w_j} ~\big|~ (x_i,y_j) \in E(\D)\big).$$
The Cohen--Macaulay property and the unmixed property of $I(\D)$ are
independent of the weight we assign to a \emph{source} vertex (i.e., a vertex
with only outgoing edges) or a \emph{sink}  (i.e.,
a vertex with only incoming edges). For this reason we shall always
assume---when studying these two properties ---that sources and sinks always have weight $1$.


The interest in edge ideals of weighted oriented graphs comes from
coding theory, in the study of Reed-Muller typed codes, as we shall
now briefly explain. Let $K = \FF_q$ be a finite field and let $G_1
\subset \dots \subset G_n$ be a nested sequence of multiplicative
subgroups of $\FF_q^* = \FF_q \setminus \{0\}.$ Consider the set of
projective points
$$\X = [(G_1 \cup \{0\}) \times \dots \times (G_n \cup \{0\})].$$
The vanishing ideal $I(\X)$ of $\X$ is generated by the set $\B$ of all binomials $x_ix_j^{w_j} - x_i^{w_i}x_j$, for $1 \le i < j \le n$, where $w_i = |G_i|+1$. Moreover $\B$ forms a Gr\"obner basis for $I(\X)$ with respect to the lexicographic order (cf. \cite{carvalho-lopez-lopez}). Let $\D$ be the oriented graph on $n$ vertices $\{x_1, \dots, x_n\}$, where $(x_i, x_j) \in E(\D)$ if and only if $i < j$, and set $\w(x_1) = 1$ and $\w(x_i) = w_i$ for $i > 1$. Then, $I(\D)$ is the initial ideal of $I(\X)$. In particular, by examining $I(\D)$, one can compute and estimate ``basic'' parameters of the Reed-Muller typed code associated to $\X$ (see \cite{carvalho-lopez-lopez,hilbert-min-dis}). For example, an interesting open problem is to compute the minimum distance of this type of linear code.

If the weight function of $\D$ is the trivial one, i.e., $\w(x) = 1$ for all $x \in V(\D)$, then $I(\D)$ recovers the usual edge ideal of its (undirected) underlying graph. Edge ideals of (undirected) graphs have been investigated extensively in the literature. In general, edge ideals of weighted oriented graphs are different from edge ideals of edge-weighted (undirected) graphs defined by Paulsen and Sather-Wagstaff \cite{PS}.

Algebraic invariants and properties of edge ideals of weighted
oriented graphs have been studied in \cite{VV, PRT}. A weighted oriented graph $\D$ is called \emph{Cohen-Macaulay} (respectively, \emph{sequentially Cohen-Macaulay}, \emph{unmixed}) if the quotient ring $R/I(\D)$ is a Cohen-Macaulay (respectively, sequentially Cohen-Macaulay, unmixed) ring. In this paper, we investigate the Cohen-Macaulayness of the edge ideal $I(\D)$ of a weighted oriented graph. Our results generalize a recent work of Gimenez, Mart\'\i nez, Simis, Villarreal and Vivares \cite{VV} from forests to arbitrary graphs, and extend a recent study of Pitones, Reyes and Toledo \cite{PRT} on the unmixedness of $I(\D)$, when the underlying graph of $\D$ is a bipartite graph, to give a complete characterization of the Cohen-Macaulayness of edge ideals for weighted oriented bipartite graphs.

Our method also leads to an affirmative answer to what was initially stated as an open question in \cite{VV} (see now \cite[Remark 3.2]{VV}). In particular, we prove the following theorem.

\begin{theorem}[Theorems \ref{cm-weighted-orientes-graphs} and \ref{linear}] \label{thm.intro}
Let $\D$ be a weighted oriented graph and let $G$ be its underlying graph. Suppose that $G$ has a perfect matching $\{x_1, y_1\}, \dots, \{x_r, y_r\}$, where $y_i$'s are leaf vertices in $G$. Then the following are equivalent:
\begin{enumerate}
\item[(a)] $\D$ is a Cohen-Macaulay weighted oriented graph;
\item[(b)] $I(\D)$ is unmixed; that is, all its associated primes have the same height;
\item[(c)] $\w(x_s) = 1$ for any edge $(x_s,y_s)$ in $\D$.
\end{enumerate}
Moreover, any of the above conditions implies
\begin{enumerate}
\item[(d)] The polarization $I(\D)^{\pol}$ has dual linear quotients;
\end{enumerate}
\end{theorem}

Our assumption in Theorem \ref{thm.intro} that $G$ has such a perfect matching comes naturally. It was shown in \cite[Theorem 3.1]{VV} that if $G$ is a forest and $\D$ is Cohen-Macaulay then $G$ must have a perfect matching $\{x_1,y_1\}, \dots, \{x_r,y_r\}$, where $y_i$'s are leaf vertices in $G$. An important application of Theorem \ref{thm.intro} is when $G$ is obtained by adding a \emph{whisker} to every vertex of a given graph. This application is inspired by a well celebrated result of Villarreal \cite{V}, where it was shown that the edge ideal of the graph obtained by adding a whisker at every vertex of a simple (undirected) graph is always Cohen-Macaulay.

The proof of Theorem \ref{thm.intro} will be broken into two parts. In Theorem \ref{cm-weighted-orientes-graphs} we establish the equivalence (a) $\Leftrightarrow$ (b) $\Leftrightarrow$ (c), and in Theorem \ref{linear} we prove (c) $\Rightarrow$ (d). It is easy to see that (a) $\Rightarrow$ (b) $\Rightarrow$ (c). Thus, it remains to show that (c) $\Rightarrow$ (a) to get Theorem \ref{cm-weighted-orientes-graphs}. To prove (c) $\Rightarrow$ (a), we apply polarization and construct an Artinian ideal whose polarization is the same as that of $I(\D)$. The proof of (c) $\Rightarrow$ (d) is more involved. We give an explicit order of the generators of the Alexander dual of $I(\D)^{\pol}$, and show that with this order the dual of $I(\D)^{\pol}$ has linear quotients. In fact, if in (d) it is known that the polarization of $I(\D)$ has pure dual linear quotients, then all four conditions would be equivalent.

If $\D$ fails the condition of Theorem \ref{thm.intro}.(c) but only at one edge $(x_s,y_s) \in E(\D)$ then $\D$ is no longer Cohen-Macaulay. We shall show that, in this case, $\D$ is sequentially Cohen-Macaulay. Specifically, we prove the following theorem.

\begin{theorem}[Theorem \ref{Clinear}] \label{thm.intro2}
Let $\D$ be a weighted oriented graph, and assume that its underlying graph $G$ has a perfect matching $\{x_1,y_1\}, \dots, \{x_r,y_r\}$, where $y_i$'s are leaf vertices in $G$. Suppose also that $(x_1,y_1) \in E(\D)$ with $\w(x_1) > 1$, while $\w(x_s) = 1$ for all $s \ge 2$ such that $(x_s,y_s) \in E(\D)$. Then $I(\D)^{\pol}$ has dual linear quotients. In particular, $\D$ is sequentially Cohen-Macaulay.
\end{theorem}

The proof of Theorem \ref{thm.intro2} follows a similar line of argument as that of Theorem \ref{linear}. The minimal generators of the Alexander dual of $I(\D)^{\pol}$, in this case, are not necessarily of the same degrees. However, we show that these generators can be obtained by multiplying appropriate variables with a minimal generator of the Alexander dual of $I(\D')^{\pol}$, where $\D'$ is the induced oriented subgraph of $\D$ on $V(\D) \setminus \{x_1,y_1\}$. Since $\D'$ satisfies condition (d) of  Theorem \ref{thm.intro}, by adapting the proof of Theorem \ref{linear} for $\D'$, we exhibit an ordering of these generators that gives linear quotients.

Our next main result gives a complete classification of Cohen-Macaulay weighted oriented graphs whose underlying graphs are bipartite. The unmixedness of $I(\D)$ in this case has been characterized in \cite[Theorem 4.17]{PRT}. Our result provides supportive evidence for \cite[Conjecture 5.5]{PRT}, which states that for any weighted oriented graph $\D$ with underlying graph $G$, the following are equivalent:
\begin{enumerate}
\item $I(\D)$ is Cohen-Macaulay; and
\item $I(\D)$ is unmixed and $I(G)$ is Cohen-Macaulay.
\end{enumerate}
We verify this conjecture for weighted oriented bipartite graphs (see Corollary \ref{cor.PRTconj}) as a corollary to the following theorem.

\begin{theorem}[Theorem \ref{cm-weighted-oriented-bipartite-graph}] \label{thm.intro1}
Let $\mathcal{D}$ be a weighted oriented bipartite graph without isolated
vertices, and let $G$ be its underlying graph.  Then
$\mathcal{D}$ is Cohen--Macaulay if and only if $G$ has a perfect matching $\{x_1,y_1\},\ldots,\{x_r,y_r\}$ such that  
the following conditions hold:
\begin{enumerate}
\item[(a)] $e_i=\{x_i,y_i\}\in E(G)$ for all $i$;
\item[(b)] if $\{x_i,y_j\}\in E(G)$, then $i\leq j$;
\item[(c)] if $\{x_i,y_j\}$, $\{x_j,y_k\}$ are in $E(G)$ and $i<j<k$,
then $\{x_i,y_k\}\in E(G)$;
\item[(d)] If $\w(y_j)\geq 2$ and
$N_\D^+(y_j)=\{x_{i_1},\ldots,x_{i_s}\}$, then $N_G(y_{i_\ell})\subset
N_\D^+(y_j)$ and all vertices of $N_\D^-(y_{i_\ell})$ have weight $1$ for $1\leq \ell\leq s$; and
\item[(e)] If $\w(x_j)\geq 2$ and
$N_\D^+(x_j)=\{y_{i_1},\ldots,y_{i_s}\}$, then $N_G(x_{i_\ell})\subset
N_\D^+(x_j)$ and all vertices of $N_\D^-(x_{i_\ell})$ have weight $1$ for $1\leq \ell\leq s$.
\end{enumerate}
\end{theorem}

The proof of Theorem \ref{thm.intro1} is an involved analysis. We first observe that $I(G) = \sqrt{I(\D)}$. Thus, if $I(\D)$ is Cohen-Macaulay then so is $I(G)$ (due to \cite[Theorem~2.6]{herzog-takayama-terai}).
It then follows from \cite[Theorem~3.4]{herzog-hibi-crit} that $G$ has a perfect matching $\{x_1,y_1\},\ldots,\{x_r,y_r\}$ such that conditions (a)-(c) hold. Since $I(\D)$ is Cohen-Macaulay, $I(\D)$ is also unmixed. Conditions (d)-(e) follow from \cite[Theorem~4.17(2)]{PRT}. 

To prove the converse statement, we make use of the following standard short exact sequence
\begin{align}
0 \rightarrow R/(I(\D):y_r)[-1] \stackrel{y_r}{\rightarrow} R/I(\D) \rightarrow R/(I(\D),y_r) \rightarrow 0 \label{eq.standardses}
\end{align}
to convert the problem to showing that both $R/(I(\D):y_r)$ and $R/(I(\D),y_r)$ are Cohen-Macaulay of appropriate dimensions. It is easy to see that $(I(\D), y_r)$ comes from the induced subgraph $\D \setminus \{x,y\}$ and, thus, is Cohen-Macaulay by induction. It remains to consider $(I(\D):y_r)$. We apply the short exact sequence (\ref{eq.standardses}) to $(I(\D):y_r)$ itself to reduce the problem to examining the ideals $((I(\D):y_r),y_r)$ and $((I(\D):y_r):y_r)$. To this end, we realize these ideals as edge ideals (together with isolated variables) of subgraphs of $\D$, and show that these subgraphs and their underlying graphs also satisfy conditions (a)-(e), and the conclusion follows by induction.

The paper is outlined as follows. In the next section, we collect notation and terminology. In Section \ref{whisker-digraphs-section}, we prove our main results characterizing the Cohen-Macaulayness of a large class of oriented graphs. In Section \ref{NCM-section}, we consider oriented graphs which fail condition (d) of Theorem \ref{thm.intro} at only one edge, and show that the polarizations of their edge ideals have dual linear quotients. In Section \ref{bipartite-section}, we prove our last main result characterizing Cohen-Macaulay edge ideals of weighted oriented bipartite graphs.

\begin{acknowledgement} This work started while the authors were at the BIRS-CMO workshop on ``Ordinary and Symbolic Powers of Ideals''. The authors would like to thank Banff International Research Station and Casa Mathem\'atica Oaxaca for their support and hospitality. The first author is partially supported by Simons Foundation and Louisiana Board of Regents.
\end{acknowledgement}


\section{Preliminaries} \label{prel-section}

In this section, we collect notation and terminology that will be used in the paper. We shall follow standard texts in the research area \cite{digraphs, Har, Mat, monalg-rev}.

Recall that $R = K[x_1, \dots, x_n]$ denotes a polynomial ring over a field $K$. For a tuple $\bsa = (a_1, \dots, a_n) \in \ZZ^n_{\ge 0}$, we shall write $\bsx^\bsa$ for the monomial $x_1^{a_1} \cdots x_n^{a_n}$ in $R$. Polarization is an essential technique in our work, so we will recall this notion following \cite{Peeva}.

\begin{definition}[\protect{\cite[Construction 21.7]{Peeva}}] \label{polarization} \quad
\begin{enumerate}
\item Let $\bsx^\bsa = x_1^{a_1} \cdots x_n^{a_n}$ be a monomial in $R$. The \emph{polarization} of $\bsx^\bsa$ is defined to be $\big(\bsx^\bsa\big)^{\pol} = (x_1^{a_1})^{\pol} \cdots (x_n^{a_n})^{\pol}$, where the operator $(\bullet)^{\pol}$ replaces $x_i^{a_i}$ by a product of distinct variables $\prod_{j = 1}^{a_i} x_{i,j}$.
\item Let $I = (\bsx^{\bsa_1}, \dots, \bsx^{\bsa_r}) \subseteq R$ be a monomial ideal. The \emph{polarization} of $I$ is defined to be the ideal $I^{\pol} = \big((\bsx^{\bsa_1})^{\pol}, \dots, (\bsx^{\bsa_r})^{\pol}\big)$ in a new polynomial ring $R^{\pol} = K[x_{i,j} ~|~ 1\le i \le n, 1 \le j \le p_i]$, where $p_i$ is the maximum power of $x_i$ appearing in $\bsx^{\bsa_1}, \dots, \bsx^{\bsa_r}$.
\end{enumerate}
\end{definition}

Our use of polarization is reflected in the following well-known result, which relates the (sequentially) Cohen-Macaulayness of an ideal with its polarization.

\begin{theorem}[\protect{\cite{Froberg} and \cite{Faridi}}] \label{PCM}
Let $I = (\bsx^{\bsa_1}, \dots, \bsx^{\bsa_r}) \subseteq R$ be a monomial ideal and, for $1 \le i \le n$, let $p_i$ be the maximum power of $x_i$ appearing in $\bsx^{\bsa_1}, \dots, \bsx^{\bsa_r}$.
\begin{enumerate}
\item Consider the sequence $\alpha = \{x_{i,j} - x_{i,1} ~|~ 1 \le i \le n, 2 \le j \le p_i\}$. Then $\alpha$ is a regular sequence in $R^{\pol}/I^{\pol}$ and
$$R^{\pol}/(I^{\pol}+(\alpha)) \simeq R/I.$$
\item $R/I$ is Cohen-Macaulay (respectively, sequentially Cohen-Macaulay) if and only if $R^{\pol}/I^{\pol}$ is Cohen-Macaulay (respectively, sequentially Cohen-Macaulay).
\end{enumerate}
\end{theorem}

Polarization allows us to reduce the study of monomial ideals to the study of squarefree monomial ideals. In doing so, we will be able to make use of combinatorial properties of squarefree monomial ideals and, in particular, to look at their associated hypergraphs.

Recall that a \emph{hypergraph} $H = (V(H), E(H))$ consists of a set $V(H)$ of distinct points, called the \emph{vertices}, and a set $E(H)$ of nonempty subsets of the vertices, called the \emph{edges}. We shall restrict ourselves to \emph{simple} hypergraphs; that is, hypergraphs with no nontrivial containments between its edges. A \emph{graph} is a hypergraph whose edges are all of cardinality 2.

Suppose that $H$ is a simple hypergraph over the vertex set $V(H) = \{x_1, \dots, x_n\}$. The \emph{edge ideal} of $H$ is defined to be
$$I(H) = \big( \prod_{x \in e} x ~|~ e \in E(H)\big) \subseteq R = K[x_1, \dots, x_n].$$
This construction gives a one-to-one correspondence between squarefree monomial ideals in $R$ and simple hypergraphs on $\{x_1,\dots, x_n\}$. Thus, we shall also denote the simple hypergraph corresponding to a squarefree monomial ideal $I \subseteq R$ by $H(I)$. We shall say that $H$ is \emph{Cohen-Macaulay} (respectively, \emph{sequentially Cohen-Macaulay}, \emph{unmixed}) if its edge ideal $I(H)$ is Cohen-Macaulay (respectively, sequentially Cohen-Macaulay, unmixed).

To investigate the (sequentially) Cohen-Macaulayness of a squarefree monomial ideal we shall examine its Alexander dual. Note that a monomial ideal $I \subseteq R$ has a unique set of minimal generators, which we shall denote by $\mathrm{gen}(I)$.

\begin{definition} Let $I \subseteq R$ be a squarefree monomial ideal. The \emph{Alexander dual} of $I$ is defined to be
$$I^\vee = \bigcap_{\bsx^\bsa \in \mathrm{gen}(I)} (x_i ~|~ x_i \text{ divides } \bsx^\bsa).$$
\end{definition}

It is an easy observation that the minimal generators of $I^\vee$ correspond to the minimal vertex covers of $H(I)$; here, for a hypergraph $H$, a subset $W$ of the vertices is called a \emph{vertex cover} if $W$ has nonempty intersection with every edge in $H$. Our use of the Alexander dual comes from the fact that the (sequentially) Cohen-Macaulayness of a squarefree monomial ideal $I$ is equivalent to the (componentwise) linearity of the minimal free resolution of $I^\vee$.

\begin{definition} Let $I \subseteq R$ be a monomial ideal. For $d \in \NN$, let $(I_d)$ denote the ideal generated by the degree $d$ elements of $I$. We say that $I$ is \emph{componentwise linear} if $(I_d)$ has a linear resolution for all $d \in \NN$.
\end{definition}

\begin{theorem}[\protect{\cite{ER} and \cite{Herzog-Hibi}}] \label{CMAD}
Let $I \subseteq R$ be a squarefree monomial ideal.
\begin{enumerate}
\item $R/I$ is Cohen-Macaulay if and only if $I^\vee$ has a linear free resolution.
\item $R/I$ is sequentially Cohen-Macaulay if and only if $I^\vee$ is componentwise linear.
\end{enumerate}
\end{theorem}

Another important technique that we shall employ is a combinatorial characterization for componentwise linear ideals.

\begin{definition} Let $I \subseteq R$ be a squarefree monomial ideal.
\begin{enumerate}
\item We say that $I$ has \emph{linear quotients} if there exists an ordering of its generators, $I = (m_1,\ldots,m_{\mu})$, such that for all $i > 1$,
$$((m_1,\ldots,m_{i-1}) : (m_i)) = (x_{k_1},\ldots,x_{k_s})$$
for some variables $x_{k_1},\ldots,x_{k_s}$. In this case, such an ordering $(m_1,\ldots,m_{\mu})$ is called a \emph{linear quotients ordering} of $I$.
\item We say that $I$ has \emph{dual linear quotients} if $I^\vee$ has linear quotients.
\end{enumerate}
\end{definition}

\begin{theorem}[\protect{\cite[Corollary 8.2.21]{Herzog-Hibi-book}}] \label{LinearQuotients}
Let $I \subseteq R$ be a squarefree monomial ideal. If $I$ has linear quotients, then $I$ is componentwise linear.
\end{theorem}

The following lemma is a technical result that we shall use. With respect to a monomial ideal $I$, we say that $x_i$ is a \emph{free variable} if $x_i$ occurs in exactly one minimal generator of $I$.

\begin{lemma}\label{FreeVar-depth}
Let $I\subset R$ be a monomial ideal and let $x_i$ be a
free variable with respect to $I$. Let $\bsx^\bsa$ be the monomial of $\mathrm{gen}(I)$ in which $x_i$ occurs. For any positive integer $m$ define
$I_m=((\mathrm{gen}(I)\setminus\{\bsx^\bsa\})\cup\{x_i^m{\bsx^\bsa}\})$. Then
$$\depth(R/I)= \depth(R/I_m).$$
In particular, $I$ is Cohen-Macaulay of height $g$ if and only if $I_m$ is Cohen-Macaulay of height $g$.
\end{lemma}

To distinguish between directed edges of an oriented graph and undirected edges of its underlying graph, we shall use the ordered pair $(x,y)$ to denote the directed edge going from $x$ to $y$, and use the unordered set $\{x,y\}$ to denote the undirected edge between $x$ and $y$.
Let $\mathcal{D}$ be an oriented graph, let $G$ be its underlying
graph, and let $v$ be a vertex of $\D$. The {\it out-neighborhood} of
$v$, denoted $N_\mathcal{D}^+(v)$, consists of all $u$ in $V(\D)$
such that $(v,u)\in E(\D)$. The {\it in-neighborhood} of $v$, denoted
$N_\mathcal{D}^-(v)$, is the set of all $u$ in $V(\D)$
such that $(u,v)\in E(\D)$. Note that $N_G(v)$, the neighbor set of $v$ in
$G$, is equal to $N_\mathcal{D}^+(v)\cup N_\mathcal{D}^-(v)$. A
non-isolated vertex $u\in V(\D)$ in an oriented graph $\D$ is a \emph{source} (respectively, a \emph{sink}) if it does not have any in-neighbors (respectively, out-neighbors).

We shall now recall a number of key notions and observations from \cite{PRT} that will be useful for our purpose.

\begin{definition}[\protect{\cite[Definition~2.3]{PRT}}] Let $\D$ be an oriented graph and let $G$ be its underlying graph.
For a vertex cover (not necessarily minimal) $C$ of $G$, define
\begin{align*}
L_1(C) & = \{x \in C ~|~ \exists (x,y) \in E(\D) \text{ with } y \not\in C\} \\
L_3(C) & = \{x \in C ~|~ N_G(x) \subseteq C\} \\
L_2(C) & = C \setminus (L_1(C) \cup L_3(C)).
\end{align*}
A vertex cover $C$ of $G$ is called a \emph{strong vertex cover} of $\D$ if $C$ is a minimal vertex cover of $G$ or for
all $x\in L_3(C)$ there is $(y,x)\in E(\D)$ such that $y\in
L_2(C)\cup L_3(C)$ with $\w(y) \ge 2$.
\end{definition}

\begin{lemma}[\protect{\cite[Theorem~4.2]{PRT}}] \label{StrongVC}
$I(\D)$ is unmixed if and only if $G$ is unmixed and $L_{3}(C)=\emptyset$ for any strong vertex cover $C$ of $\D$.
\end{lemma}

\begin{theorem}[{\rm\cite[Theorem~3.11]{PRT}}]\label{prt-main} Let $\mathcal{D}$ be a weighted
oriented graph. Then $\mathfrak{p}$ is an associated prime of
$I(\mathcal{D})$ if and only if $\mathfrak{p}=(C)$ for some strong
vertex cover $C$ of $\mathcal{D}$.
\end{theorem}


\section{Oriented graphs with perfect matchings}\label{whisker-digraphs-section}

In this section, we shall prove our main results. Particularly, we
will give various equivalent algebraic and combinatorial
characterizations for the Cohen-Macaulayness of an oriented graph.
Our results extend \cite[Theorem 3.1]{VV} from the case where the
underlying graph is a forest to a more general situation. Our method
also gives an affirmative answer to an open problem stated in \cite[Remark 3.2]{VV}.

It follows from \cite[Theorem 3.1]{VV} that if the edge ideal of an oriented graph $\D$ is Cohen-Macaulay and its underlying graph $G$ is a forest, then $G$ has a perfect matching $$\{x_1,y_1\}, \dots, \{x_r,y_r\},$$ where $y_1, \dots, y_r$ are leaf vertices. Our results address the class of oriented graphs whose underlying graphs have such a perfect matching.

Our first main theorem gives an easy combinatorial characterization of Cohen-Macaulay weighted oriented graphs.

\begin{theorem}\label{cm-weighted-orientes-graphs}
Let $\D$ be a weighted oriented graph and let $G$ be its underlying graph. Suppose that $G$ has a
perfect matching $\{x_1,y_1\},\ldots,\{x_r,y_r\}$, where $y_i$'s are leaf vertices. Then the following are equivalent:
\begin{itemize}
\item[(a)] $\D$ is a Cohen-Macaulay weighted oriented graph;
\item[(b)] $I(\D)$ is unmixed; that is, all its associated primes have the same height;
\item[(c)] $\w(x_s)=1$ for any edge $(x_s,y_s)$ of $\D$.
\end{itemize}
\end{theorem}

\begin{proof}
We will show that (a) $\Rightarrow$ (b) $\Rightarrow$ (c) $\Rightarrow$ (a).

It is a well-known fact that if $R/I(\D)$ is Cohen-Macaulay, then all the associated
primes of $I(\D)$ have the same height (see, for example, \cite[Theorem
17.3]{Mat}). Thus, (a) implies (b).

To see that (b) implies (c), assume that $I(\D)$ is unmixed and $(x_s,y_s) \in E(\D)$.
Suppose, by contradiction, that $\w(x_s)>1$. By our initial assumption that a source vertex has weight 1, $x_{s}$ is not a source. Thus, there is at least one incoming edge at $x_s$. Since $y_i$'s are leaf vertices in $G$, there exists $l$ with
$(x_l,x_s)\in E(\D)$ (which, in particular, implies that $(x_s,x_l) \not\in E(\D)$ and $\deg_G(x_s) \ge 2$). Observe that
$$C=(\{x_{1},\dots,x_{r}\}\setminus \left\{x_l\right\}) \cup \left\{y_l,y_s\right\}$$
is a (non-minimal) vertex cover of $G$. Observe also that $L_{3}(C)=\{y_{s}\}$ and $x_s \not\in L_1(C)$. This implies that $x_s \in
L_2(C)$. Therefore, $C$ is a strong vertex cover of $\D$ because $(x_s,y_s) \in E(\D)$ and $\deg_G(x_s) \ge 2$. Hence, it follows from Lemma \ref{StrongVC} that $I(\D)$ is not unmixed, a contradiction. That is, we must have $\w(x_s) = 1$, and so (b) implies (c).

It remains to prove that (c) implies (a). To achieve this we will use polarization techniques. For simplicity of notation, set $\w_i=\w(x_i)$ for $i=1,\dots,r$. By Lemma~\ref{FreeVar-depth}, we may also assume that $\w(y_i)=1$ for all
$i = 1, \dots, r$.

Let $\D[X]$ be the induced oriented subgraph of $\D$ on the vertex set
$X=\{x_1,\dots,x_r\}$. Let $\H$ be a copy of $\D[X]$ on the vertex set $Z = \{z_1, \dots, z_r\}$ obtained by replacing $x_i$ with $z_i$, for $i = 1, \dots, r$, and let $S = K[z_1, \dots, z_r]$ (in particular, the weight of $z_i$ in $\H$ is $\w_i$). Consider the ideal
$$J = I(\H) + (z_1^{w_1+1}, \dots, z_r^{\w_r+1}) \subseteq S.$$

By re-indexing if necessary, we may assume that for some $k \le r$, $\w_i \ge 2$ for $1 \le i < k$ and $\w_i = 1$ for $k \le i \le r$. It follows from the assumption in (c) that $(y_i,x_i) \in E( \D)$ for $i=1,\dots,k$. Thus,
$$
I(\D)=(y_1x_1^{\w_1},\dots, y_kx_k^{\w_k},x_{k+1}y_{k+1},\dots,x_ry_r) + I(\D[X]).
$$

Observe that in the polarization construction of $J$, the variables $z_1, \dots, z_r$ are replaced by $z_{1,1}, \dots, z_{r,1}$, and
$$\big(z_i^{\w_i+1}\big)^{\pol} = z_{i,1} \dots z_{i,\w_i+1} \text{ for } i = 1, \dots, r.$$
On the other hand, in the polarization construction of $I(\D)$, the variables $x_1, \dots, x_r$ are replaced by $x_{1,1}, \dots, x_{r,1}$, and
$$\big(y_ix_i^{\w_i}\big)^{\pol} = y_{i,1}x_{i,1} \dots x_{i,\w_i} \text{ for } i = 1, \dots, r.$$
Now, consider the variable identification $\phi: S^{\pol} \rightarrow R^{\pol}$ given by $z_{i,j} \mapsto x_{i,j}$ for $1 \le i \le r$ and $1 \le j \le \w_i$ and $z_{i,\w_i+1} \mapsto y_{i,1}$ for $1\le i \le r$. Then
$$I(\D)^{\pol} = \phi(J^{\pol}).$$
Moreover, since $S/J$ is an Artinian ring, $S^{\pol}/J^{\pol}$ is Cohen-Macaulay by Theorem \ref{PCM}. Hence, $R^{\pol}/I(\D)^{\pol}$ is Cohen-Macaulay and, by Theorem \ref{PCM} again, $R/I(\D)$ is Cohen-Macaulay.
\end{proof}

In general, without assuming condition (c) of
Theorem~\ref{cm-weighted-orientes-graphs}, the polarizations of
$I(\D)$ and $J$ are not necessarily isomorphic. The next example
illustrates how this can occur.

\begin{example}\label{aug11-17}
Consider the oriented graph $\D$ on the vertices $\{x_1, x_2, y_1, y_2\}$ with edges $(y_1,x_1)$, $(x_1,x_2)$, and $(x_2,y_2)$, whose weight function is given by $\w(x_1)=2$, $\w(x_2)=3$, and $\w(y_1) = \w(y_2) = 1$. The ideal $J$ in the proof of Theorem~\ref{cm-weighted-orientes-graphs} is $J=\{z_1^3, z_1z_2^{3},z_2^4\}$, and
$I(\D)=(y_1x_1^2,x_1x_2^3,x_2y_2)$. It is easy to see that the polarizations of $J$ and
$I(\D)$ are not isomorphic. In fact, $J$ is Artinian (and so Cohen-Macaulay), and $\D$ is not Cohen-Macaulay.
\end{example}

\begin{remark}
It is worth noting that there are alternate approaches to showing the Cohen-Macaulayness of $I(\D)$. One such approach is through grafted simplicial complexes. See \cite{Faridi2} for relevant definitions.
\end{remark}

Our next main result exhibits that the equivalent conditions (a), (b), and (c) of Theorem~\ref{cm-weighted-orientes-graphs} in fact imply an even stronger condition, that is, $I(\D)$ has dual linear quotients. Before giving the full precise statement, let us collect a number of important properties of vertex covers of oriented graphs.

\begin{notation}\label{assumptions}
For simplicity of notation, for the remainder of this section, by a weighted, oriented graph with whiskers we will mean that $\D$ is a weighted oriented graph whose underlying graph $G$ has a perfect matching $\{x_1,y_1\}, \dots, \{x_r,y_r\}$ where $y_i$ is a leaf for all $i$. Let $\D[X]$ denote the induced oriented subgraph of $\D$ on $X = \{x_1, \dots, x_r\}$. Recall also that in the polarization construction of $I(\D)$, we have $I(\D)^{\pol} \subseteq R^{\pol} = K[x_{i,j}, y_{i,k} ~|~ 1 \le i \le r, 1 \le j \le \w(x_i), 1 \le k \le \w(y_i)]$.
\end{notation}

\begin{lemma}\label{MVCPolarized}
Let $\D$ be a weighted oriented graph with whiskers and let $J = I(\D)^{\pol}$. Let $C$ be a minimal vertex cover of the hypergraph $H(J)$ associated to $J$, and set
\begin{align*}
C_1 & = C \cap \{x_{1,1}, \dots, x_{r,1}\} \\
C_2 & = C \cap \{x_{i,j} ~|~ 1 \le i \le r, 2 \le j \le \w(x_i)\} \\
C_3 & = C \cap \{y_{i,k} ~|~ 1 \le i \le r, 1 \le k \le \w(y_i)\}.
\end{align*}
Then $C = C_1 \cupdot C_2 \cupdot C_3$ is a disjoint union such that for each $i$ there is a $j\geq 1$ with either $x_{i,j}$ or $y_{i,j}$ in $C$. Moreover, for each $i$ such that $x_{i,1} \not\in C_1$, each of $C_2$ and $C_3$ must contain at most one of the $x_{i,j}$'s and $y_{i,k}$'s, respectively.
\end{lemma}

\begin{proof}
It is easy to see, from the definition of $C_1, C_2$ and $C_3$, that $C = C_1 \cupdot C_2 \cupdot C_3$ is a disjoint union. By the definition of $\D$, for each $i$, either $x_i^{\w(x_i)}y_i$ or $x_iy_i^{\w(y_i)}$ is in $I(\D)$. Thus $C$ must contain $x_{i,j}$ or $y_{i,j}$ for some $j$ to cover the polarization of this edge. Suppose that $x_{i,j} \in C$. By the polarization construction, any minimal generator of $J = I(\D)^{\pol}$ divisible by $x_{i,j}$ must also be divisible by $x_{i,t}$ for all $t \le j$. That is, any edge of $H(J)$ containing $x_{i,j}$ must also contain $x_{i,t}$ for all $t \le j$. Thus, $x_{i,t} \not\in C$ for any $t < j$ by the minimality of $C$. A similar observation holds for any $y_{i,j} \in C$. Thus $C_2$ and $C_3$ must each contain at most one of the $x_{i,j}$'s and $y_{i,k}$'s respectively.
\end{proof}

\begin{remark}\label{EdgesCausingEmbeddedPrimes}
A careful examination of the structure of $\D$ can reveal more about the set $C_2$.
Suppose $x_{s,1} \not\in C_1$ for some $s$. If there exists an edge $(x_l, x_s)$ of $\D$ with $x_{l,1} \not\in C_1$ then since $C$ must cover the edge $\big(x_lx_s^{\w(x_s)}\big)^{\pol}$ in $H(J)$, we must have $\w(x_s) > 1$ and $x_{s,j} \in C$ for some $2 \leq j \leq \w(x_s)$.
Note that if in addition $(x_s,y_s) \in I(\D)$, then $y_{s,k} \in C_3$ for some $k\geq 1$ is needed to cover the edge $(x_sy_s)^{\pol}$ in $H(J)$. Thus in the final statement of Lemma~\ref{MVCPolarized} it is possible that both of $C_2$ and $C_3$ contain an $x_{s,j}$ and $y_{s,k}$ respectively.

Conversely, if both of $C_2$ and $C_3$ contain an $x_{s,j}$ and $y_{s,k}$ respectively for some $s$, then because $C$ is minimal, in addition to the polarization of generator corresponding to the whisker edge, which is covered by $y_{s,k}$, there must be an edge of $H(J)$ covered by $x_{s,j}$ but not $y_{s,k}$. Such an edge is the polarization of a generator of $I(\D)$ of the form $x_lx_s^{\w(x_s)}$. Thus for some $l$ there is an edge $(x_l,x_s) \in \D$ with $x_{l,1} \not\in C_1$.
\end{remark}

If $\D$ is Cohen-Macaulay then the following lemma shows that for any $i$ such that $x_{i,1} \not\in C$, $C$ cannot contain a polarizing variable for both $x_i$ and $y_i$.

\begin{lemma}\label{MVCPolarizedCM}
Let $\D$ be a Cohen-Macaulay weighted oriented graph with whiskers. Consider a minimal vertex cover $C$ of $H(I(\D)^{pol})$, and suppose $C = C_1 \cupdot C_2 \cupdot C_3$ as in Lemma \ref{MVCPolarized}. Then $|C|=r$ and for each $i$ such that $x_{i,1} \not\in C_1$, either $x_{i,j} \in C_2$ for some $2 \leq j \leq \w(x_i)$ or $y_{i,k} \in C_3$ for some $1 \leq k \leq \w(y_i)$, but not both.
\end{lemma}

\begin{proof}
For each $i$, either $(x_i,y_i)$ or $(y_i,x_i)$ is in $E(\D)$. Suppose that $(x_i,y_i)\in \D$. Then $w(x_i)=1$ and $x_{i,1}y_{i,1}\dots y_{i,\w(y_i)} \in I(\D)^{\pol}$, so either $x_{i,1} \in C$ or $y_{i,k}$ is in $C$ for some $k$ but not both since $y_i$ is a leaf of the underlying graph $G$. Similarly, if $(y_i,x_i)\in E(\D)$, then either $y_{i,1} \in C$ or $x_{i,j}$ is in $C$ for some $j$, but not both. Thus, if $x_{i,1} \not\in C$, then either $x_{i,j} \in C$ for some $2 \leq j \leq \w(x_i)$ or $y_{i,k} \in C$ for some $1 \leq k \leq \w(y_i)$.

Now, since $\D$ is Cohen-Macaulay, then $\D$ is unmixed by Theorem \ref{cm-weighted-orientes-graphs}. Observe that $\{x_{1,1}, \dots , x_{r,1}\}$ is a vertex cover of $H(I(\D)^{\pol})$, since every edge contains $x_i$ for some $i$, and it is minimal by the existence of the perfect matching in $G$. Thus, every minimal vertex cover of $H(I(\D)^{\pol})$ must have exactly $r$ elements. Hence, for each $i$ such that $x_{i,1} \not\in C$, only one of the $x_{i,j}$'s or $y_{i,k}$'s can be in $C$.
\end{proof}

\begin{corollary}\label{CompareEntries}
Under the same hypotheses as Lemma~\ref{MVCPolarizedCM}, if $C_1 \cupdot C_2 \cupdot C_3$ and $C_1' \cupdot C_2' \cupdot C_3'$ are minimal vertex covers of $H(I(\D)^{\pol})$, and $C_2=C_2'$ and $C_3=C_3'$, then $C_1=C_1'$.
\end{corollary}

\begin{proof}
The conclusion follows from the fact that for each $i$, precisely one of the containments $x_{i,1}\in C_1$, or $x_{i, j}\in C_2$ for some $j \geq 2$, or $y_{i, j}\in C_3$ for some $j \geq 1$ occurs.
\end{proof}

\begin{remark}\label{AssGeneralForm}
By \cite[Corollary 2.6]{Faridi} the associated primes of the monomial ideal $I(\D)$ can be found by depolarizing the associate primes of $I(\D)^{\pol}$, all of which are minimal since $I(\D)^{\pol}$ is square-free.
By Lemma~\ref{MVCPolarized}, the minimal primes of $I(\D)$ will be depolarizations of minimal primes of $J=I(\D)^{\pol}$ for which for every $i$ at most one of $C_2$ and $C_3$ contains an $x_{i,j}$ or $y_{i,k}$ respectively. These primes have height $r$ and are of the form $C_x\cupdot C_y$ where $C_x \subseteq \{ x_1, \ldots , x_r\}$ forms a minimal vertex cover of the undirected underlying induced graph on $\{x_1, \ldots , x_r\}$ and $C_y = \{ y_i \, | \, x_{i,1} \not\in C_1\}$. The embedded associated primes of $I(\D)$ will be depolarizations of minimal primes of $J$ described in Lemma~\ref{MVCPolarized} where $C_2$ and $C_3$ both contain an $x_{i,j_i}$ and $y_{i, k_i}$ respectively for one or more $i$. These embedded primes can be described by looking at the directed edges of $\D$ as described in Remark~\ref{EdgesCausingEmbeddedPrimes}. All embedded primes of $I(\D)$ have the form $C \cup \{ y_{i_1}, \ldots , y_{i_t} \}$ where $C$ is a minimal prime of $I(\D)$ and for each $1 \leq s \leq t$, $d(x_{i_s}) > 1$, $(x_{i_s}, y_{i_s}) \in E(\D)$, and there is an edge $(x_{j_s}, x_{i_s})\in E(\D)$  for some $x_{j_s} \not\in C$.

\end{remark}

The following lemma allows us to construct new minimal vertex covers of $H(I(\D)^{\pol})$ from given ones. This will be useful later on in ordering the generators of the Alexander dual of $I(\D)^{\pol}$ to show linear quotients.

\begin{lemma}\label{moveDown}
Let $\D$ be a weighted oriented graph with whiskers. Suppose that $C$ is a minimal vertex cover of $H(I(\D)^{\pol})$.  If $y_{i,j} \in C$ for some $j \geq 2$, then $C \setminus \{y_{i, j} \} \cup \{ y_{i, j-1} \}$ is also a minimal vertex cover of $H(I(\D)^{\pol})$. A similar statement holds if $x_{i, j} \in C$ for some $j \geq 3.$
\end{lemma}

\begin{proof}
It can be seen from the polarization construction that if a minimal generator of $I(\D)^{\pol}$ contains $y_{i,j}$ then it also contains $y_{i, j-1}$. Thus, $C \setminus \{y_{i, j} \} \cup \{ y_{i, j-1} \}$ is also a vertex cover of $H(I(\D)^{\pol})$.  When $j \geq 2$, the minimality of $C \setminus \{y_{i, j} \} \cup \{ y_{i, j-1} \}$ comes immediately from that of $C$, since there is a unique edge of $\D$ containing $y_i$. The statement for $x_{i,j}$ follows in the same line of arguments as that for $y_{i,j}$.
\end{proof}

It is worth noting that the minimality of $C \setminus \{x_{i,2}\} \cup \{x_{i,1}\}$ in general may not be true. For instance, in Example~\ref{aug11-17}, $C=(x_{1,2}, x_{2,2}, y_2)$ is a minimal vertex cover of $H(I(\D)^{pol})$, but $C\setminus\{x_{1,2}\} \cup \{x_{1,1} \}$ is not minimal.

If $\D$ is known to be a Cohen-Macaulay oriented graph, we can push Lemma \ref{moveDown} slightly further.

\begin{corollary}\label{moveDownCM}
Let $\D$ be a Cohen-Macaulay weighted oriented graph with whiskers. Suppose that $C$ is a minimal vertex cover of $H(I(\D)^{\pol})$. If $x_{i,j}$ or $y_{i,j}$ is in $C$ for some $j \ge 2$ then $C \setminus \{x_{i,j}\} \cup \{x_{i,j-1}\}$ or $C \setminus \{y_{i, j} \} \cup \{ y_{i, j-1} \}$, respectively, is also a minimal vertex cover of $H(I(\D)^{\pol})$.
\end{corollary}

\begin{proof}
We only need to prove the statements for $x_{i,2}$. Suppose that $x_{i,2} \in C$.  By the polarization construction, every minimal generator of $I(\D)^{\pol}$ that contains $x_{i,2}$ also contains $x_{i,1}$. Thus, $x_{i,1} \not\in C$ and replacing $x_{i,2}$ with $x_{i,1}$ results in a vertex cover of $H(I(\D)^{\pol})$.

Now, since $\D$ is Cohen-Macaulay, all minimal vertex covers of $H(I(\D)^{\pol})$ must have exactly $r$ elements (see Lemma~\ref{MVCPolarizedCM}). Hence, the replacement described does not change the minimality of a minimal vertex cover.
\end{proof}

\begin{corollary}\label{moveDownCM2}
Let $\D$ be a Cohen-Macaulay weighted oriented graph with whiskers. Suppose $C$ is a minimal vertex cover of $H(I(\D)^{\pol})$. If $y_{i,k} \in C$ for some $k \ge 1$ then $C \setminus \{y_{i,k}\} \cup \{x_{i,j}\}$ is also a minimal vertex cover of $H(I(\D)^{\pol})$ for any $j \ge 1$.
\end{corollary}

\begin{proof} As in Corollary \ref{moveDownCM}, the minimality of $C \setminus \{y_{i,k}\} \cup \{x_{i,j}\}$ if it is a vertex cover would follow from the Cohen-Macaulayness of $\D$. Thus, it remains to show that $C \setminus \{y_{i,k}\} \cup \{x_{i,j}\}$ is a vertex cover of $H(I(\D)^{\pol})$. That is, any edge of $H(I(\D)^{\pol})$ that contains $y_{i,k}$, for some $k \ge 1$ must also contain $x_{i,j}$ for all $j \ge 1$. Since $y_i$ is a leaf, this statement is clearly true if $\w(x_i) = 1$.

Consider the case where $\w(x_i) > 1$. Then by Theorem \ref{cm-weighted-orientes-graphs}, we must have $(x_i,y_i) \not\in E(\D)$ and $(y_i,x_i) \in E(\D)$. That is, the only generator of $I(\D)$ containing $y_i$ must be $y_ix_i^{\w(x_i)}$ (which also implies that the given $y_{i,k}$ must indeed be $y_{i,1}$). Clearly, $(y_ix_i^{\w(x_i)})^{\pol}$ contains $x_{i,j}$ for all $j \ge 1$.
\end{proof}

We are now ready to prove our next main theorem, which states that the equivalent conditions in Theorem \ref{cm-weighted-orientes-graphs} in fact imply that $I(\D)^{\pol}$ has dual linear quotients. Generally speaking, having dual linear quotients can be a powerful tool in showing that an ideal is Cohen-Macaulay by using  Theorem \ref{CMAD} and \ref{LinearQuotients}; however, $(I(\D)^{\pol})^{\vee}$ must be pure to follow this path. Since $(I(\D)^{\pol})^{\vee}$ being pure is trivially equivalent to $I(\D)^{\pol}$ being unmixed, which has already been shown to be equivalent to being Cohen-Macaulay in this setting, the implication of interest is that the condition on the vertex weights forces dual linear quotients.

\begin{theorem}\label{linear}
Let $\D$ be a weighted oriented graph and let $G$ be its underlying graph. Suppose that $G$ has a
perfect matching $\{x_1,y_1\},\dots,\{x_r,y_r\}$, where $y_i$'s are leaf vertices. If $\w(x_{\ell})=1$ for any edge $(x_{\ell},y_{\ell})$ of $\D$, then $I(\D)^{\pol}$ has dual linear quotients.
\end{theorem}

\begin{proof} For simplicity of notation, let $J = I(\D)^{\pol}$ and let $H = H(J)$. Recall that the minimal generators of the Alexander dual $J^\vee$ are obtained from the minimal vertex covers of $H$. By Theorems \ref{PCM} and \ref{cm-weighted-orientes-graphs}, $J$ is Cohen-Macaulay and unmixed. Thus, all minimal vertex covers of $H$ have exactly $r$ elements and are of the form described in Lemma~\ref{MVCPolarizedCM}.

To show that $J$ has dual linear quotients, that is, $J^\vee$ has linear quotients, we shall order the minimal generators $M_1, \dots, M_s$ of $J^\vee$ lexicographically, where
\begin{align*}
x_{1,1} & >x_{1,2}> \ldots >x_{1,\w(x_{1})}>y_{1,1}>y_{1,2}> \ldots >y_{1,\w(y_{1})}>x_{2,1}>x_{2,2}>\ldots >x_{2,\w(x_{2})}\\
 & >y_{2,1}> \ldots >y_{2,\w(y_{2})}> \ldots >x_{r,1}>x_{r,2}> \ldots >x_{r,\w(x_{r})}>y_{r,1} > \ldots > y_{r,\w(y_{r})}.
\end{align*}
Observe that by Lemma~\ref{MVCPolarized}, for any $i$, each minimal generator of $J^\vee$ is divisible by $x_{i, j}$ or $y_{i, j}$ for at most one $j$. Thus, $M_1 = \prod_{i=1}^r x_{i,1}$ is the first generator under this ordering.

For each generator $M$ of $J^{\vee}$, order the variables dividing $M$ as above. By Lemma~\ref{MVCPolarizedCM}, $M$ has the form
\begin{equation}\label{product}
M=a_1a_2\cdots a_r, 
\end{equation}
where for each $i$, $a_i=x_{i,j_i}$ for some $j_i\geq 1$ or $a_i=y_{i,k_i}$ for some $k_i \geq 1$.
Consider an arbitrary $t \le s$ (where $s$ is the number of monomial generators of $J$) and any $1 \le u <t$. Using the format of (\ref{product}), write $M_t=a_1a_2\cdots a_r$ and $M_u=b_1b_2\cdots b_r$.  Let $i_u$ be the least integer such that $a_{i_u}\not= b_{i_u}$. By the chosen ordering, one of the following must occur:
\begin{enumerate}
\item $a_{i_u}=x_{i,j_i}$ and $b_{i_u}=x_{i,j}$ for some $j < j_i$;
\item $a_{i_u}=y_{i,k_i}$ and $b_{i_u}=y_{i,k}$ for some $k < k_i$;
\item $a_{i_u}=y_{i,k_i}$ and $b_{i_u}=x_{i,j}$ for some $j$.
\end{enumerate}
Thus $M_u : M_t$ is a multiple of  $b_{i_u}$. Moreover, by applying Corollaries \ref{moveDownCM} and \ref{moveDownCM2}, $b_{i_u}M_t/a_{i_u}$ is a minimal generator of $J^\vee$ and $b_{i_u}M_t/a_{i_u} < M_t$ in our ordering, with $b_{i_u}M_t/a_{i_u}:M_t=b_{i_u}$. Hence, it follows that $(M_1, \dots, M_{t-1}):M_t$ is generated by a subset of the variables. Thus $J^{\vee}=(I(\D)^{pol})^{\vee}$ has linear quotients.
\end{proof}

\begin{corollary}
With the assumptions of Theorem \ref{linear}, the Stanley-Reisner simplicial complex of $I(\D)^{pol}$ is shellable and the Alexander dual of $I(\D)^{\pol}$ has a linear free resolution.
\end{corollary}

An important application of our main results is the following statement, which is inspired by the main theorem of \cite{V}, which showed that the edge ideal of the graph obtained by adding a whisker to every vertex of an arbitrary graph is always Cohen-Macaulay.

\begin{corollary} Let $G'$ be an arbitrary simple graph on the vertex set $\{x_1, \dots, x_r\}$, and let $G$ be obtained by adding whiskers $\{x_1, y_1\}, \dots, \{x_r,y_r\}$ to the vertices of $G'$. Let $\D$ be a weighted oriented graph with $G$ as its underlying graph. Then the following are equivalent:
\begin{enumerate}
\item[(a)] $\D$ is a Cohen-Macaulay weighted oriented graph;
\item[(b)] $I(\D)$ is unmixed; that is, all its associated primes have the same height;
\item[(c)] $\w(x_s) = 1$ for any edge $(x_s,y_s)$ in $\D$.
\end{enumerate}
Moreover, any of the above conditions implies
\begin{enumerate}
\item[(d)] The polarization $I(\D)^{\pol}$ has dual linear quotients.
\end{enumerate}
\end{corollary}
As before, if $I(\D)^{\vee}$ is pure and has linear quotients, then condition $(d)$ becomes equivalent to $(a),(b),$ and $(c)$.


\section{Non-Cohen-Macaulay oriented graphs} \label{NCM-section}

Our results in Section \ref{whisker-digraphs-section} show that $\D$ is Cohen-Macaulay if and only if $\w(x_s) = 1$ for any edge $(x_s,y_s) \in E(\D)$. When the later condition fails, we of course do not expect $\D$ to be Cohen-Macaulay. In this section, we shall prove that if the condition of Theorem~\ref{cm-weighted-orientes-graphs}(c) only fails at one vertex $x_s$, then $\D$ is sequentially Cohen-Macaulay and $I(\D)^{\pol}$ has dual linear quotients.

We begin with a lemma which allows us to describe the minimal generators of the Alexander dual of $I(\D)^{\pol}$ when the condition of Theorem~\ref{cm-weighted-orientes-graphs}(c) fails at only one vertex $x_s$.

\begin{lemma} \label{MVCPolarizedNotCM}
Let $\D$ be a weighted oriented graph, and assume that its underlying graph $G$ has a perfect matching $\{x_1,y_1\}, \dots, \{x_r,y_r\}$, where $y_i$'s are leaf vertices in $G$. Suppose also that $(x_1,y_1) \in E(\D)$ with $\w(x_1) > 1$, while $\w(x_i) = 1$ for all $i \ge 2$ such that $(x_i,y_i) \in E(\D)$. Let $C$ be a minimal vertex cover of $H(I(\D)^{\pol})$ and assume that $C = C_1 \cupdot C_2 \cupdot C_3$ as in Lemma \ref{MVCPolarized}. Then the following statements hold:
\begin{enumerate}
\item If $x_{i,1} \not\in C_1$ then each of $C_2$ and $C_3$ contains at most one of the $x_{i,j}$'s and $y_{i,j}$'s, respectively. Moreover, if $i \not= 1$ then either $x_{i,j} \in C_2$ for some $2 \le j \le \w(x_i)$ or $y_{i,k} \in C_3$ for some $1 \le k \le \w(y_i)$, but not both.
\item $|C| = r \text{ or } r+1$. Furthermore, $|C| = r+1$ if and only if there exists an $x_s$ (with $s \not= 1$) such that $(x_s,x_1) \in E(\D)$ and $\{x_{1,1}, x_{s,1}\} \not\subseteq C_1$. In this case, $C_2$ must contain $x_{1,j}$ for some $2 \le j \le \w(x_1)$ and $C_3$ must contain  $y_{1,k}$ for some $1 \le k \le \w(y_1)$.
\end{enumerate}
\end{lemma}

\begin{proof} The first statement of (1) follows the same line of arguments as in Lemma \ref{MVCPolarizedCM}. To prove the second statement of (1), let $\D'$ be the induced oriented subgraph of $\D$ over the vertex set $\{x_2, \dots, x_r, y_2, \dots, y_r\}$. Then $\D'$ satisfies the condition of Theorem~\ref{cm-weighted-orientes-graphs}(c). Thus, $\D'$ is Cohen-Macaulay and the assertion again follows from the proof of Lemma \ref{MVCPolarizedCM}.

We shall now prove (2). Let $C' = C \cap \{x_2, \dots, x_r, y_2, \dots, y_r\}$ and let $\D'$ be as above. It is clear that $C'$ is a vertex cover of $\D'$. Since $C$ is a minimal vertex cover of $\D$ and there are whiskers at $x_i$'s, we must have either $x_{i,k}\in C$ for some $1\leq k \leq \w(x_i)$ or $y_{i,k'} \in C$ for some $1\leq k' \leq \w(y_i)$. By statement (1), for $i>1$, either $x_{i,k} \in C' $ for some $1\leq k \leq \w(x_i)$ or $y_{i,k'} \in C'$ for some $1\leq k' \leq \w(y_i)$ but not both. Hence $|C'| = r-1$. By applying Theorem~\ref{cm-weighted-orientes-graphs} and Lemma \ref{MVCPolarizedCM} to $\D'$, noting that all the minimal vertex covers of $\D'$ have size $r-1$, we deduce that $C'$ is a minimal vertex cover of $\D'$.

Observe that $C' \cup \{x_{1,1}\}$ is a vertex cover of $H(I(\D)^{\pol})$. Thus, if $x_{1,1} \in C$ then we must have $C = C' \cup \{x_{1,1}\}$, whence $|C| = r$.

Suppose now that $x_{1,1} \not\in C$. Notice that if $(x_1,x_s) \in E(\D)$ and $\w(x_s) > 1$, for some $s \ge 2$, then the hypothesis forces $(y_s,x_s) \in E(\D)$. This implies that if  $x_{s,1}\not\in C$ then  $x_{s,j} \in C$ for some $2 \le j \le \w(x_s)$ and  $y_{s,k} \not\in C$ for any $1 \le k \le \w(y_s)$. Moreover, if for any edge $(x_s,x_1) \in E(\D)$, where $s \ge 2$, we always have $x_{s,1} \in C$, then since $C$ must contain $y_{1,k}$ for some $1 \le k \le \w(y_1)$ and $C' \cup \{y_{1,k}\}$ is a vertex cover of $H(I(\D)^{\pol})$, we have $C = C' \cup \{y_{1,k}\}$, and again, $|C| = r$.

It remains to consider the case where there exists an $x_s$ ($s \not= 1$) such that $(x_s,x_1) \in E(\D)$ and $\{x_{1,1}, x_{s,1}\} \not\subseteq C$. In this case, in order to cover $(x_sx_1^{\w(x_1)})^{\pol}$, $C$ must contain $x_{1,j}$ for some $2 \le j \le \w(x_1)$, and in order to cover $(x_1y_1^{\w(y_1)})^{\pol}$, $C$ must contain $y_{1,k}$ for some $1 \le k \le \w(y_1)$. Furthermore, for such $j$ and $k$, $C' \cup \{x_{1,j}, y_{1,k}\}$ is a vertex cover of $H(I(\D)^{\pol})$. Thus, $C = C' \cup \{x_{1,j}, y_{1,k}\}$ for some $2 \le j \le \w(x_1)$ and $1 \le k \le \w(y_1)$ and $|C| = r+1$.  \end{proof}

\begin{remark}\label{Form-of-covers}
With the statement of Lemma \ref{MVCPolarizedNotCM}, we know a minimal vertex
cover of $H(I(\mathcal{D})^{pol})$ has size either $r$ or $r+1$.
When it has size $r$ then for each $i$, we have exactly one of the
following happens: $x_{i,1}\in C_{1}$; $x_{i,j}\in C_{2}$ for $1<j<\omega(x_{i})$;
$y_{i,k}\in C_{3}$ for $1\leq k\leq\omega(y_{i})$. Moreover when
$x_{1,1}\notin C_{1}$ and the size of $C$ is $r$, then for
all $1<s\leq r$ if $(x_{s},x_{1})\in E(\mathcal{D})$ then $x_{s,1}\in C_{1}$.
If the size of $C$ is $r+1$, then we must have $x_{1,1}\notin C_{1}$
and for some $1<s\leq r$, $(x_{s},x_{1})\in E(\mathcal{D})$ and
$x_{s,1}\notin C_{1}$. See also Remark~\ref{AssGeneralForm}.
\end{remark}

We are now ready to present our next main result, which shows that if condition $(c)$ of Theorem~\ref{cm-weighted-orientes-graphs} is violated at a single vertex, then the resulting ideal, while no longer Cohen-Macaulay, remains sequentially Cohen-Macaulay.

\begin{theorem}\label{Clinear}
Let $\D$ be a weighted oriented graph, and assume that its underlying graph $G$ has a perfect matching $\{x_1,y_1\}, \dots, \{x_r,y_r\}$, where $y_i$'s are leaf vertices in $G$. Suppose also that $(x_1,y_1) \in E(\D)$ with $\w(x_1) > 1$, while $\w(x_s) = 1$ for all $s \ge 2$ such that $(x_s,y_s) \in E(\D)$. Then $I(\D)^{\pol}$ has dual linear quotients. In particular, $\D$ is sequentially Cohen-Macaulay.
\end{theorem}

\begin{proof} The second statement follows from the first statement and Theorems \ref{PCM}, \ref{CMAD} and \ref{LinearQuotients}.

We now show that $I(\D)^{\pol}$ has dual linear quotients. We shall order the generators of $I(\D)^{\pol}$ by listing degree $r$ generators first and then degree $(r+1)$ generators. Among generators of the same degrees, we shall use the same lexicographic ordering as that given in the proof of Theorem~\ref{linear}. Suppose that $M_1, \dots, M_v$ are the minimal generators of $I(\D)^{\pol}$ given in this ordering, in which, for some $u$ with $u < v$, $M_1, \dots, M_u$ are of degree $r$ and $M_{u+1}, \dots, M_v$ are of degree $(r+1)$.

Let $\D'$ be the induced oriented subgraph of $\D$ on the vertex set $\{x_2, \dots, x_r, y_2, \dots, y_r\}$ and consider an arbitrary generator $M_t$ for $1 \le t \le v$. Without loss of generality, we may also assume that for some fixed $w,p$ with $1 \le w \le p \le r$, $(x_s,x_1) \in E(\D)$ for all $2 \le s \le w$, $(x_1, x_s) \in E(\D)$ for all $w+1 \le s \le p$ and for $s>p$, $\{x_1, x_s\}\not\in E(G)$.

\noindent{\bf Case 1: $t \le u$.} In this case, $\deg(M_t) = r$ and, by Remark~\ref{Form-of-covers}, $M_{t}=x_{1,1}M'$ or $M_{t}=y_{1,k}M'$, where $M'$ is a minimal generator of $I(\D')^{\pol}$. Notice that $M_t$ cannot be of form $M_{t}=x_{1,k}M'$ for $k>1$, otherwise the edge $(x_1,y_1)\in E(\D)$ is not covered by the corresponding minimal vertex cover.  If $M_{t}=x_{1,1}M'$ then by our ordering, $M_{i}$ must contain
$x_{1,1}$ for all $i<t$. That is $M_{i}=x_{1,1}N_{i}$ where $N_{i}$
is a minimal generators of $I(\D')^{\pol}$ that comes
before $M'$ in the corresponding ordering. Thus, it follows from the
proof of Theorem \ref{linear} for $\D'$ that $(M_{1},....,M_{t-1}):M_{t}$
is generated by a subset of the variables.

Suppose that $M_{t}=y_{1,k}M'$. Then, by Remark~\ref{Form-of-covers}, this is the case only if $x_{s,1} \big| M'$ for any $2 \le s \le w$ (i.e., when $(x_s,x_1) \in E(\D)$). Consider any $i < t$. By our ordering, $\deg(M_i) = r$ and, thus, $M_{i}$ is either $x_{1,1}M''$, where $M''$ is a minimal generator of $I(\D')^{\pol}$, or $M_{i}=y_{1,k'}M''$, where $k' \le k$ and $M''$ is a minimal generator of $I(\D')^{\pol}$. This implies that $M_{i}:M_{t}$ is contained in the ideal generated by $x_{1,1},y_{1,k'}$ and $M'':M'$. It is easy to see that $x_{1,1}M'$ and $y_{1,k'}M'$( if $k'<k$) are minimal
generators of $I(\D)^{\pol}$, which come before $M_{t}$ in our ordering. Moreover, if $k'=k$ then $M'' > M'$ in the ordering of $I(\D')^{\pol}$ and, conversely, for any such minimal generator $M'' > M'$ of $I(\D')^{\pol}$, the assumption on $M'$ forces $x_{s,1} \big| M''$ for all $2 \le s \le w$, whence $y_{1,k}M''$ is also a minimal generator of $I(\D)^{\pol}$ which comes before $M_t$. Hence, together with the conclusion of Theorem \ref{linear} for $\D'$, we have that $(M_{1},...,M_{t-1}):M_{t}$ is generated by a subset of the variables.

\noindent{\bf Case 2: $t > u$.} In this case, $\deg(M_{t})=r+1$ and, by Lemma \ref{MVCPolarizedNotCM},
$M_{t}=x_{1,j}y_{1,k}M'$, where $2\le j \le \omega(x_{1})$, $1\le k\le \omega(y_{1})$, and $M'$ is a minimal generator of $I(\D')^{\pol}$ such that there exists $s \ge 2$ (and necessarily $s \le w$) for which $(x_s,x_1) \in E(\D)$ and $x_{s,1}\nmid M'$. Consider $M_{i}$ for $i<t$.

If $i\leq u$ then $M_{i}$ is either $x_{1,1}M''$ or $y_{1,k'}M''$ where $M''$ is a minimal generator of $I(\D')^{\pol}$. Thus, $M_i:M_t$ is contained in the ideal generated by $x_{1,1}, y_{1,k'}$ and $M'':M'$. It is easy to see that $x_{1,1}M'$ and $x_{1,j}y_{1,k'}M'$ (if $k'<k$) are minimal generators of $I(\D)^{\pol}$, which comes before $M_{t}$ in our ordering.

Observe now that if $M_i = y_{1,k'}M''$ then, by the proof of Lemma \ref{MVCPolarizedNotCM}, for any $2 \le l \le w$, since $(x_l,x_1) \in E(\D)$, $M''$ must contain $x_{l,1}$. This implies that, in this case, $M'' > M'$ in our ordering of the generators of $I(\D')^{\pol}$ (since $x_{s,1} \nmid M'$). Furthermore, if $k' = k$ then $M'' > M'$ in the ordering for $I(\D')^{\pol}$ and, conversely, for any minimal generator $M'' > M'$ of $I(\D')^{\pol}$, either $y_{1,k}M''$ (if $x_{l,1} \big| M''$ for all $2 \le l \le w$) or $x_{1,j}y_{1,k}M''$ (otherwise) is a minimal generator of $I(\D)^{\pol}$, which comes before $M_t$ in our ordering. Hence, together with Theorem \ref{linear} for $\D'$, we have that $(M_{1},...,M_{t-1}):M_{t}$ is generated by a subset of the variables.

If $i > u$ then $M_i = x_{1,j'}y_{1,k'}M''$, where $2 \le j' < j \le \w(x_1)$, or $j' = j$ and $1 \le k' < k \le \w(y_1)$, or $j' = j$, $k' = k$ and $M''$ comes before $M'$ in the ordering for $I(\D')^{\pol}$. In this case, $M_i:M_t$ is contained in the ideal generated by $x_{1,j'}$, $y_{1,k'}$ and $M'':M'$. Observe that if $j' < j$ then clearly $x_{1,j'}y_{1,k}M'$ is a minimal generator of $I(\D)^{\pol}$ which comes before $M_t$. If $j' = j$ and $k' < k$ then $x_{1,j}y_{1,k'}M'$ is a minimal generator of $I(\D)^{\pol}$ which comes before $M_t$. If $j' = j, k' = k$ then $M''$ comes before $M'$ in the ordering of $I(\D')^{\pol}$. On the other hand, for any minimal generator $M'' > M$ of $I(\D')^{\pol}$, either $y_{1,k}M''$ (if $x_{l,1} \big| M''$ for all $2 \le l \le w$) or $x_{1,j}y_{1,k}M''$ (otherwise) is a minimal generator of $I(\D)^{\pol}$, which comes before $M_t$. Thus, we again have that $(M_1, \dots, M_{t-1}):M_t$ is generated by a subset of the variables.
\end{proof}

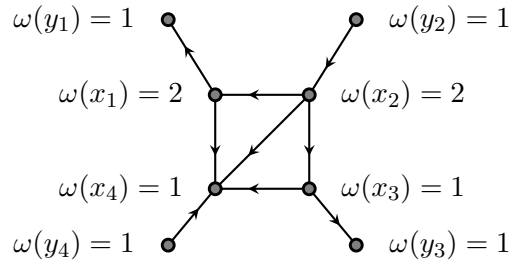
\begin{figure}[h]
\begin{center}
	\begin{tikzpicture}[line width=1.1pt,scale=1.25]
	\tikzstyle{every node}=[inner sep=0pt, minimum width=4.5pt]
	\draw[thick, directed]  (0,0)--(-0.5,.8); 
	\draw[thick, directed]  (1.5,0.8)--(1,0); 
	\draw[thick, directed]  (1,-1)--(1.5,-1.6); 
	\draw[thick, directed](-0.5,-1.6)--(0,-1); 
        \draw[thick, directed] (0,0)--(0,-1); 
        \draw[thick, directed] (1,0)--(0,0); 
        \draw[thick, directed]  (1,0)--(0,-1); 
	\draw[thick, directed] (1,0)--(1,-1); 
	\draw[thick, directed]  (1,-1)--(0,-1); 
		
	\draw (0,0) node (x1) [draw, circle, fill=gray] {};
	\draw (1,0) node (x2) [draw, circle, fill=gray] {};
	\draw (1,-1) node (x3) [draw, circle, fill=gray] {};
	\draw (0,-1) node (x4)[draw, circle, fill=gray] {};
	\draw (-0.5,0.8) node (y1) [draw, circle, fill=gray] {};
        \draw (1.5,0.8) node (y2) [draw, circle, fill=gray] {};
	\draw (1.5, -1.6) node (y3) [draw, circle, fill=gray] {};
	\draw (-0.5, -1.6) node (y4) [draw, circle, fill=gray] {};
	
	\node at (-1,0) {\small $\omega(x_{1})=2$};
	\node at (2,0) {\small $\omega(x_{2})=2$};
	\node at (2,-1) {\small $\omega(x_{3})=1$};
	\node at (-1,-1) {\small $\omega(x_{4})=1$};
	\node at (-1.5,0.8) {\small $\omega(y_{1})=1$};
	\node at (2.5,0.8) {\small $\omega(y_{2})=1$};
    \node at (2.5,-1.6) {\small $\omega(y_{3})=1$};
    \node at (-1.5,-1.6) {\small $\omega(y_{4})=1$};
	\end{tikzpicture}
\caption{Sequentially Cohen-Macaulay weighted oriented graph $\D$.}\label{graph2}
\end{center}
\end{figure}

\begin{example}
Let $\D$ be the weighted oriented graph depicted in Figure \ref{graph2}. Then
\[I(\D) = (x_1y_1,x_1^2x_2,x_1x_4,x_2^2y_2,x_2x_3,x_2x_4,x_3y_3,x_3x_4,x_4y_4).\]
The graph $\D$ fails condition (c) of Theorem \ref{cm-weighted-orientes-graphs} at $x_1$, and so $I(\D)$ is not Cohen-Macaulay. On the other hand, $I(\D)$ is sequentially Cohen-Macaulay by Theorem \ref{Clinear}.
\end{example}

\begin{remark}
In general, we expect $I(\D)^{\pol}$ to have linear quotients (and so, $I(\D)$ is sequentially Cohen-Macaulay) whenever $G$ has a perfect matching $\{x_1, y_1\}, \dots, \{x_r, y_r\}$ in which the $y_i$'s are leaf vertices. The argument, however, has proved to be much more subtle and involved than that of Theorem \ref{Clinear}; and we shall leave that to our future work.
\end{remark}


\section{Cohen--Macaulay weighted oriented
bipartite graphs}\label{bipartite-section}

In this section, we address \cite[Conjecture 5.5]{PRT} for weighted oriented bipartite graphs. Particularly, we give a complete classification for the Cohen-Macaulay property of edge ideals of weighted oriented bipartite graphs. For this class of graphs, the unmixedness of their edge ideals was already characterized in \cite{PRT}.

Our last main result of the paper is stated as follows.

\begin{theorem}\label{cm-weighted-oriented-bipartite-graph}
Let $\mathcal{D}$ be a weighted oriented bipartite graph without isolated
vertices, and let $G$ be its underlying graph.  Then
$\mathcal{D}$ is Cohen--Macaulay if and only if $G$ has a perfect matching $\{x_1,y_1\},\ldots,\{x_r,y_r\}$ such that
 the following conditions hold:
\begin{enumerate}
\item[(a)] $e_i=\{x_i,y_i\}\in E(G)$ for all $i$;
\item[(b)] if $\{x_i,y_j\}\in E(G)$, then $i\leq j$;
\item[(c)] if $\{x_i,y_j\}$, $\{x_j,y_k\}$ are in $E(G)$ and $i<j<k$,
then $\{x_i,y_k\}\in E(G)$;
\item[(d)] If $\w(y_j)\geq 2$ and
$N_\D^+(y_j)=\{x_{i_1},\ldots,x_{i_s}\}$, then $N_G(y_{i_\ell})\subset
N_\D^+(y_j)$ and all vertices of $N_\D^-(y_{i_\ell})$ have weight $1$ for $1\leq \ell\leq s$; and
\item[(e)] If $\w(x_j)\geq 2$ and
$N_\D^+(x_j)=\{y_{i_1},\ldots,y_{i_s}\}$, then $N_G(x_{i_\ell})\subset
N_\D^+(x_j)$ and all vertices of $N_\D^-(x_{i_\ell})$ have weight $1$ for $1\leq \ell\leq s$.
\end{enumerate}
\end{theorem}

\begin{proof} Since the Cohen--Macaulay property is additive on
graded ideals in disjoint sets of variables
(see \cite[Lemma~4.1]{V}), we may assume that $G$ is a connected
bipartite graph. According to \cite[Corollary~6]{VV} the
Cohen--Macaulay property of $\mathcal{D}$ is dependent
only on knowing which vertices have weight greater than one and not
on the actual weights used. Thus, we may also assume that $\w(v)=2$ for
any vertex $v$ of $\mathcal{D}$ with $\w(v)>1$.

\noindent $(\Rightarrow$) We observe that the radical of $I(\mathcal{D})$ is $I(G)$. Therefore,
$I(G)$ is Cohen--Macaulay by \cite[Theorem~2.6]{herzog-takayama-terai}.
It follows from \cite[Theorem~3.4]{herzog-hibi-crit} that $G$ has a
perfect matching $\{x_1,y_1\},\ldots,\{x_r,y_r\}$ such that (a)-(c)
hold. Observe further that $I(\D)$ is Cohen-Macaulay, and so $I(\D)$ is unmixed (see \cite[Corollary~3.1.17]{monalg-rev}). Thus, (d)-(e) follow from \cite[Theorem~4.17(2)]{PRT}.

\noindent ($\Leftarrow$) Let $R=K[x_1,\ldots,x_r,y_1,\ldots,y_r]$ be a polynomial ring associated to $\D$. We proceed by induction on $r$. The case $r=1$ is
clear because $I(\mathcal{D})$ is generated by $x_1y_1$. Suppose that $r \ge 2$. It suffices to show that $(I(\D)\colon
y_r)$ and $(I(\D),y_r)$ are
Cohen--Macaulay ideals of dimension $r$.
Indeed, from the exact sequence
\begin{equation}\label{apr4-18}
0\rightarrow R/(I(\mathcal{D})\colon
y_r)[-1]\stackrel{y_r}{\rightarrow}
R/I(\mathcal{D})\rightarrow R/(I(\mathcal{D}),y_r)\rightarrow 0,
\end{equation}
and using the depth lemma \cite[Lemma~2.3.9]{monalg-rev}, we obtain that $I(\mathcal{D})$ is
Cohen--Macaulay.

Let $\D\setminus\{x_r,y_r\}$ be the digraph obtained from $\D$ by
removing the vertices $x_r$ and $y_r$ and all edges containing at
least one of them. The ideal $(I(\D),y_r)$ is Cohen--Macaulay of dimension $r$ because
$(I(\D),y_r)$ is equal to $(I(\D\setminus\{x_r,y_r\}),y_r)$ and
$I(\D\setminus\{x_r,y_r\})$ is Cohen--Macaulay by induction. Thus, the
proof reduces to showing that $(I(\D)\colon y_r)$ is Cohen--Macaulay.
From the exact sequence
$$
0\rightarrow R/((I(\D)\colon y_r)\colon y_r)[-1]\stackrel{y_r}{\rightarrow}
R/(I(\D)\colon y_r)\rightarrow R/((I(\D)\colon y_r),y_r)\rightarrow 0,
$$
to show that $(I(\D)\colon y_r)$ is Cohen--Macaulay,
we need only show that $((I(\D)\colon y_r)\colon y_r)$ and
$((I(\D)\colon y_r),y_r)$ are Cohen--Macaulay of dimension $r$.
Let $V^i$ be the set of vertices of $\D$ of weight $i$ for
$i=1,2$. Set $V'=N_\D^+(y_r)\cap V^1$ and $V''=N_\D^+(y_r)\cap V^2$.

\noindent{ \sc Case 1:} Assume that $\w(y_r)=2$. In
particular, $y_r$ cannot be a source, $(x_r,y_r)\in E(\D)$ (by (d)), and
$x_ry_r^2\in I(\D)$. Note that if $x_i\in N_\D^-(y_r)$
(respectively, $x_i\in V^\prime$, $x_i\in V''$), then $x_iy_r^2\in I(\D)$ (respectively,
$x_iy_r\in I(\D)$, $x_i^2y_r\in I(\D)$). Therefore, we have the equalities
\begin{align}
(I(\D)\colon y_r) &= (x_ry_r)+(x_iy_r\vert\, x_i\in
N_\D^-(y_r))+(V')+(x_i^2\vert\, x_i\in V'')+I(\D\setminus
V'),\label{apr17-1-18}\\
((I(\D)\colon y_r),y_r) &= (V',y_r)+(x_i^2\vert\, x_i\in
V'')+I(\D\setminus A)=(I(\D)\colon x_ry_r),\label{apr16-18}\\
((I(\D)\colon y_r)\colon y_r) &= (N_\D^-(y_r),V')+
(x_i^2\vert\, x_i\in V'')+I(\D\setminus B),\label{apr17-18}
\end{align}
where $A=V'\cup\{y_r\}$ and $B=N_\D^-(y_r)\cup V'$. Note
that $x_r\in B$ since $x_r \in N_\D^-(y_r)$.
Consider the weighted oriented graph $\H$ whose vertex set and edge set are
$V(\H)=V(\D\setminus A)$
and
$$
E(\H)=E(\D\setminus A)\setminus\{e\in E(\D)~\vert\, e=(y_j,x_i)\
\mbox{ and }\
x_i\in V''\},
$$
respectively.

We first show that $V_Y'=\{y_i\vert\, x_i\in V'\}$ is a
set of isolated vertices of $\H$, i.e., they are vertices that are
not in any edge of $\H$. Take $y_i\in V_Y'$ and assume that
$y_i$ is not isolated. Then there is $x_j\in V(\H)$, $x_j\notin A$,
such that $\{x_j,y_i\}$ is an edge of the underlying graph $H$ of
$\H$. By (b)-(c), $\{x_j,y_r\}$ is an edge of the underlying
graph $G$ of $\D$. Assume that $(y_r,x_j)\in E(\D)$. As $x_j\notin V'$,
we have $\w(x_j)=2$ and $x_j\in V''$. Hence, $(y_i,x_j)$ cannot be an edge of $\H$; that
is, $(x_j,y_i)\in E(\H)$. As $y_i\in N_\D^+(x_j)$, by (e),
$N_G(x_i)\subset N_\D^+(x_j)$ and, in particular, $y_r\in N_\D ^+(x_j)$
because $y_r\in N_\D^-(x_i)$; that is $(x_j,y_r)\in E(\D)$, a
contradiction. We may now assume $(x_j,y_r)\in E(\D)$. As $x_i\in
N_\D^+(y_r)$, by (d), $N_G(y_i)\subset N_\D^+(y_r)$. In
particular, as $x_j\in N_G(y_i)$, we get $x_j\in N_\D^+(y_r)$, that
is, $(y_r,x_j)$ is an edge of $\D$, a contradiction. This proves that
$V_Y'$ is isolated in $\H$.

We next show that $V_Y''=\{y_i\vert\, x_i\in V''\}$ is a
set of isolated vertices of $\H$. Take $y_i\in V_Y''$ and assume that
$y_i$ is not isolated. Then, there is $x_j\in V(\H)$, $x_j\notin A$,
such that $\{x_j,y_i\}$ is an edge of $H$. Assume that $\w(x_j)=2$. As $x_i\in N_\D^+(y_r)$, by
(d), we have $N_G(y_i)\subset N_\D^+(y_r)$. In particular, since
$x_j\in N_G(y_i)$, we get $x_j\in N_\D^+(y_r)$ and $x_j\in V''$. Thus,
$(y_i,x_j)$ cannot be an edge of $\H$; that is,
$(x_j,y_i)$ is an edge of $\H$. Then, $y_i\in N_\D^+(x_j)$ and, by (e), $N_G(x_i)\subset N_\D^+(x_j)$. In particular, as $y_r\in N_\D^-(x_i)$, we have
$(x_j,y_r)\in E(\D)$, contradicting that $x_j\in N_\D^+(y_r)$.
We may now assume that $\w(x_j)=1$. Then,
$x_j\notin N_\D^+(y_r)$ because $x_j\notin A$ and, by (c), $x_j\in
N_\D^-(y_r)$; that is, $(x_j,y_r)\in E(\D)$. Now, $x_i\in N_\D^+(y_r)$
because $x_i\in V''$. Then, by (d), $N_G(y_i)\subset N_\D^+(y_r)$.
In particular, as $x_j\in N_G(y_i)$, we have $x_j\in N_\D^+(y_r)$,
that is, $(y_r,x_j)$ is in $E(\D)$, a contradiction.

The vertex $x_r$ is also an isolated vertex of $\H$ and $y_r$ is not a
vertex of $\H$. Setting $L$ equal to $((I(\D)\colon y_r),y_r)$, from
Eq.~(\ref{apr16-18}) and noticing that the ideal $I(\D\setminus A)+(x_i^2\vert\, x_i\in
V'')$ is equal to $I(\H)+(x_i^2\vert\, x_i\in
V'')$,
we obtain
\begin{equation}\label{apr21-18-1}
L=((I(\D)\colon y_r),y_r)=(V',y_r)+I(\H)+(x_i^2\vert\, x_i\in
V'')=(I(\D)\colon x_ry_r).
\end{equation}

The vertex set of $\D$ has a decomposition
$V(\D)=V(\D\setminus A)\cup A$ with $|A|=|V'|+1$. To compute the
heights of $L$ and $I(\H)$ notice the following decomposition
\begin{equation}\label{apr21-18}
V(\H)=V(\D\setminus A)=\left(\bigcup_{x_i\notin V'\cup
V''\cup\{x_r\}}\hspace{-4mm}\{x_i,y_i\}\right)\cup\{x_r\}\cup V''\cup
V_Y'\cup V_Y''
\end{equation}
and recall that $V_Y'\cup V_Y''\cup\{x_r\}$ is a set of isolated
vertices of $\H$. The set $C_Y$ of all $y_i$ such that $y_i\notin
V_Y'\cup V_Y''$ is contained in $V(\H)$ and is a minimal vertex cover of $\H$. Indeed, if
$\{x_j,y_i\}$ is an edge of $H$, then $y_i$ is not isolated in $\H$,
and consequently $y_i\notin V_Y'\cup V_Y''$. Thus, $C_Y$ is a vertex
cover of $\H$. From Eq.~(\ref{apr21-18}) it follows that any vertex
cover of $\H$ has at least $|C_Y|=r-|V'|-|V''|-1$ elements.
Hence, ${\rm ht}(I(\H))=|C_Y|=r-|N_\D^+(y_r)|-1$ and
$$
{\rm ht}(L)=|V'|+1+{\rm ht}(I(\H))+|V''|=r.
$$
Thus, by Eq.~(\ref{apr21-18-1}), to show that $L=((I(\D)\colon y_r),y_r)$ is Cohen--Macaulay
of dimension $r$ it suffices to show that the ideal $L_1=I(\H)+(x_i^2\vert\, x_i\in
V'')$ is Cohen--Macaulay.

Let $\F$ be the weighted oriented bipartite graph with
$$
V(\F)=\left(\bigcup_{x_i\notin V'\cup
V''\cup\{x_r\}}\{x_i,y_i\}\right)\cup V''\cup V_Y''
$$
and $E(\F)=E(\H)\cup\{(x_i,y_i)\vert x_i\in V''\}$, where the
vertices of $V_Y''$ and $V''$ have weight $1$ (see the discussion below). Clearly, $\F$ has
a perfect matching, satisfies (b), and ${\rm ht}(I(\F))=r-|V'|-1$.

To see that $\F$ satisfies (c), consider the underlying graph $F$
of $\F$ and take $\{x_i,y_j\}$, $\{x_j,y_k\}$ in $E(F)$ with $i<j<k$.
Then, as $\D$  satisfies (c), $\{x_i,y_k\}$ is in $E(G)$. A vertex
$x_\ell$ is in $V(\F)$ if and only if
$y_\ell$ is in $V(\F)$. Thus, the vertices $x_i$ and $y_k$ are in
$V(\F)$. If $(x_i,y_k)$ is in $E(\D)$, then $(x_i,y_k)$ is in
$E(\H)\subset E(\F)$ and $\{x_i,y_k\}$ is an edge of $F$. Thus, we may
assume $(y_k,x_i)\in E(\D)$. If $\w(x_i)=1$, then $(y_k,x_i)$ is in
$E(\H)\subset E(\F)$. Now assume that $\w(x_i)=2$. Then, $(x_i,y_j)\in
E(\D)$ because $\{x_i,y_j\}\in E(F)$. As $y_j\in N_\D^+(x_i)$, by (e),
one has $N_G(x_j)\subset N_\D^+(x_i)$. In particular, since
$y_k\in N_G(x_j)$, we get $y_k\in N_\D^+(x_i)$. Thus, $(x_i,y_k)\in
E(\D)$, a contradiction. Altogether it follows that $\{x_i,y_k\}$ is
in $E(F)$ and $\F$ satisfies (c).

Consider the partial polarization $L_1^{\pol}$ obtained from $L_1$
by polarizing all variables $x_i^2$ with $x_i\in V''$. All vertices $x_i$ of $\H$ that are in $V''$
are sources by the construction of $\H$. Hence, all variables in $V''$ occur
in the minimal generating set of $I(\H)$ with exponent $1$. Therefore,
up to an appropriate change of variables, $I(\F)$ is the partial polarization
$L_1^{\pol}$ of $L_1$. Since $\D$ satisfies (a) and (d)-(e),
by \cite[Theorem~4.17]{PRT}, we get that $I(\D)$ is unmixed. Hence,
as $L$ is equal to $(I(\D)\colon x_ry_r)$ and $I(\D)$ is unmixed, we
get that $L$ is unmixed and so are $L_1$ and $L_1^{\pol}$.
Since $L_1^{\pol}$ is $I(\F)$, applying \cite[Theorem~4.17]{PRT}, we get that $\F$
satisfies (d)-(e). Altogether $\F$ satisfies (a)-(e).
Therefore, applying induction to $\F$, $I(\F)$ is
Cohen--Macaulay; that is, $L_1^{\pol}$ is Cohen--Macaulay.
Thus, $L_1$ is Cohen--Macaulay, as required.

The arguments for  $((I(\D):y_r):y_r)$ are similar. Consider the weighted oriented graph $\H_1$ whose vertex set and edge set are
$V(\D\setminus B)$
and
$$
E(\H_1)=E(\D\setminus B)\setminus\{e\in E(\D)\vert\, e=(y_j,x_i)\
\mbox{ and }\
x_i\in V''\},
$$
respectively.

We claim that $V_B'=\{y_i\vert\, x_i\in B\}$ is a
set of isolated vertices of $\H_1$. Take $y_i\in V_B'$. We proceed by
contradiction, assuming there is an edge $\{x_j,y_i\}$ in the
underlying graph $H_1$ of $\H_1$. Then, $x_j\notin B$. Since $x_i\in
B$, by (c), $\{x_j,y_r\}\in E(G)$. Then, as $x_j\notin B$, $x_j\in
N_\D^+(y_r)$ and $\w(x_j)=2$. Thus, $(x_j,y_i)\in E(\H_1)$; that is,
$y_i\in N_\D^+(x_j)$. Then, by (e), we get $N_G(x_i)\subset
N_\D^+(x_j)$. In particular, as $y_r\in N_G(x_i)$, we have
$(x_j,y_r)\in E(\D)$, contradicting to the fact that $x_j\in N_\D^+(y_r)$.

We next claim that the set $V_Y''=\{y_i\vert\, x_i\in V''\}$ is also isolated in
$\H_1$. Take $y_i\in V_Y''$. 
Clearly one has $i\neq r$
because $x_i\in V''$ but $x_r$ is a leaf of the underlying graph and thus a sink or a source, so $w(x_r)=1$. Assume that $y_i$ is not isolated in $\H_1$ and pick an edge
$\{x_j,y_i\}$ in $H_1$. Then, $y_i\notin A$ and
consequently $\{x_j,y_i\}$ is an edge of $H$, a contradiction because
we have previously shown that any vertex of $V_Y''$ is isolated in $H$.

The vertex $y_r$ is an isolated vertex of $\H_1$ because $y_r\in V_B'$
and $x_r$ is not a
vertex of $\H_1$. Setting $N$ equal to $((I(\D)\colon y_r)\colon y_r)$, from
Eq.~(\ref{apr17-18}) and noticing that $I(\D\setminus B)+(x_i^2\vert\, x_i\in
V'')$ is equal to $I(\H_1)+(x_i^2\vert\, x_i\in
V'')$,
we obtain
$$
N=((I(\D)\colon y_r)\colon y_r)=(B)+I(\H_1)+(x_i^2\vert\, x_i\in
V'').
$$

As $\D$ satisfies (a) and (d)-(e),
by \cite[Theorem~4.17]{PRT}, $I(\D)$ is unmixed of height $r$. Hence,
the ideal $N_1=I(\H_1)+(x_i^2\vert\, x_i\in
V'')$ has height $r-|N_\D^-(y_r)|-|V'|$. Thus, to show that
the ideal $N=((I(\D)\colon y_r)\colon y_r)$ is Cohen--Macaulay of height $r$ it
suffices to show that the ideal $N_1=I(\H_1)+(x_i^2\vert\, x_i\in
V'')$ is Cohen--Macaulay. We have the decomposition
\begin{equation}\label{apr21-18-2}
V(\H_1)=V(\D\setminus B)=\left(\bigcup_{x_i\notin B\cup
V''}\hspace{-2mm}\{x_i,y_i\}\right)\cup V_B'\cup V_Y''\cup V'',
\end{equation}
where $V_B'\cup V_Y''$ is a set of isolated
vertices of $\H_1$.

Let $\F_1$
be the weighted oriented bipartite graph with
$$
V(\F_1)=\left(\bigcup_{x_i\notin B\cup
V''}\{x_i,y_i\}\right)\cup V_Y''\cup V''
$$
and $E(\F_1)=E(\H_1)\cup\{(x_i,y_i)\vert x_i\in V''\}$, where the
vertices of $V_Y''$ and $V''$ have weight $1$. It is seen that $\F_1$
satisfies (a)-(c). Consider the partial polarization $N_1^{\pol}$
obtained from $N_1$
by polarizing all variables $x_i^2$ with $x_i\in V''$. All vertices
$x_i$ of $\H_1$ that are in $V''$
are sources by construction of $\H_1$. Hence, all variables in $V''$ occur
in the minimal generating set of $I(\H_1)$ with exponent $1$. Therefore,
up to an appropriate change of variables, $I(\F_1)$ is the partial polarization
$N_1^{\pol}$ of $N_1$. Since $\D$ satisfies (a) and (d)-(e),
by \cite[Theorem~4.17]{PRT}, we get that $I(\D)$ is unmixed. Hence,
as $N$ is equal to $((I(\D)\colon y_r)\colon y_r)$ and $I(\D)$ is unmixed, we
get that $N$ is unmixed and so are $N_1$ and $N_1^{\pol}$.
Since $N_1^{\pol}$ is $I(\F_1)$, applying \cite[Theorem~4.17]{PRT} we
get that $\F_1$
satisfies (d)-(e). Altogether $\F_1$ satisfies (a)-(e).
Therefore, applying induction to $\F_1$, $I(\F_1)$ is
Cohen--Macaulay, that is, $N_1^{\pol}$ is Cohen--Macaulay.
Thus, $N_1$ is Cohen--Macaulay, as required.

\noindent{\sc Case 2:} Assume that $\w(y_r)=1$.
The proof of this case is similar to that of Case 1. We shall only identify the differences between the two cases and leave the details for the interested reader.
The equalities~(\ref{apr17-1-18})-(\ref{apr17-18}), for Case 2, will be replaced with the following equalities
\begin{eqnarray}
(I(\D)\colon y_r) & = &(N_\D^-(y_r))+(V')+(x_i^2\vert\, x_i\in V'')+I(\D\setminus (V'\cup N_\D^-(y_r))),\label{apr26-1}\\
((I(\D)\colon y_r),y_r) &=& (V',N_\D^-(y_r),y_r)+(x_i^2\vert\, x_i\in
V'')+I(\D\setminus A'),\label{apr26-2}\\
((I(\D)\colon y_r)\colon y_r) &=& (V',N_\D^-(y_r))+
(x_i^2\vert\, x_i\in V'')+I(\D\setminus B),\label{apr26-3}
\end{eqnarray}
where $A'=V'\cup N_\D^-(y_r)\cup\{y_r\}$ and $B=V'\cup N_\D^-(y_r)$. Observe that (\ref{apr26-2}) and (\ref{apr26-3}) are both similar to (\ref{apr17-18}) and, thus, the proof for $((I(\D):y_r):y_r)$ in Case 1 shall be employed to prove the Cohen-Macaulayness of both $((I(\D):y_r), y_r)$ and $((I(\D):y_r):y_r)$ in Case 2.
\end{proof}

It can be seen that if $\mathcal{D}$ is a Cohen--Macaulay weighted
oriented graph, then $I(\mathcal{D})$ is unmixed
(see \cite[Corollary~3.1.17]{monalg-rev}) and
$\sqrt{I(\mathcal{D})}$ is Cohen--Macaulay
(see \cite[Theorem~2.6]{herzog-takayama-terai}).
It is an open question whether the converse is true.

\begin{conjecture}[\protect{\cite[Conjecture~5.5]{PRT}}] \label{conj.PRT} Let $\D$ be a weighted
oriented graph and let $G$ be its underlying graph. If $I(\D)$ is
unmixed and $I(G)$ is Cohen--Macaulay, then $I(\D)$ is Cohen--Macaulay.
\end{conjecture}

One may ask whether or not an unmixed monomial ideal $I$ of graph
type (i.e., a monomial ideal $I$ whose
radical is generated by square-free monomials of degree $2$) with
Cohen--Macaulay radical is Cohen--Macaulay.

As a consequence of Theorem \ref{cm-weighted-oriented-bipartite-graph}, we show that Conjecture \ref{conj.PRT} holds for weighted oriented
bipartite graphs.

\begin{corollary} \label{cor.PRTconj}
Let $\D$ be a weighted oriented bipartite graph
without isolated vertices and
let $G$ be its underlying graph. If $I(\D)$ is unmixed and $I(G)$ is
Cohen--Macaulay, then $I(\D)$ is Cohen--Macaulay.
\end{corollary}

\begin{proof} By \cite[Theorem~3.4]{herzog-hibi-crit},
conditions (a)-(c) hold if and only if $I(G)$ is Cohen--Macaulay. On
the other hand, by \cite[Theorem~4.17]{PRT}, conditions (a) and
(d)-(e) hold if and only if $I(\D)$ is unmixed. 
\end{proof}

\begin{remark}
Under the conditions of Theorem~\ref{cm-weighted-oriented-bipartite-graph}, if $(y_t,x_t)\in E(\mathcal{D})$ for some $t$ (respectively, $(x_t,y_t)\in
E(\mathcal{D})$ for some $t$), then $\w(y_t)=1$ (respectively,
$\w(x_t)=1$); that is, the tail of any edge in the perfect matching has
weight $1$. To see this, suppose that $(y_t,x_t)\in E(\D)$ (the case $(x_t,y_t)\in E(\D)$ is similar). We
proceed by contradiction, assuming that $\w(y_t)\geq 2$. In particular,
$y_t$ is not a source; that is, there is $(x_i,y_t)\in E(\D)$. Since
$x_t\in N_\D^+(y_t)$, by (d), $N_G(y_t)\subset N_\D^+(y_t)$. Hence,
as $x_i\in N_G(y_t)$, we get $x_i\in N_\D^+(y_t)$, a contradiction.
\end{remark}

\begin{example}\label{vila-example}
The edge ideal $I(\D)$ of the weighted oriented graph $\D$ of
Fig.~\ref{graph1} is Cohen--Macaulay and is generated by
$x_1^2y_1,x_1y_2,x_1y_3,x_1^2y_4,x_2^2y_2,x_2y_3,x_3y_3,x_4y_4^2$.
\end{example}
\begin{figure}[h]
\begin{center}
	\begin{tikzpicture}[line width=1.1pt,scale=0.95]
	\tikzstyle{every node}=[inner sep=0pt, minimum width=4.5pt]
        \draw[thick, directed] (-1.5,-0.9)--(-1.5,.9); 
	 \draw[thick, directed]  (-1.5,1)--(0.42,-0.94); 
	\draw[thick, directed]  (-1.5,1)--(2.407,-1.055); 
        \draw[thick, directed]  (4.41,-0.96)--(-1.43,1); 
	 \draw[thick, directed] (0.5,-1)--(0.5,0.9); 
	  \draw[thick, directed] (0.5,1)--(2.467,-0.901); 
	  \draw[thick, directed] (2.5,-1)--(2.5,0.9); 
	\draw[thick, directed] (4.5,1)--(4.5,-0.9); 
   
	\draw (-1.5,-1) node (v1) [draw, circle, fill=gray] {};
	\draw (0.5,-1) node (v2) [draw, circle, fill=gray] {};
	\draw (2.5,1) node (x3) [draw, circle, fill=gray] {};
	\draw (0.5,1) node (x2)[draw, circle, fill=gray] {};
	\draw (-1.5,1) node (x1) [draw, circle, fill=gray] {};
        \draw (2.5,-1) node (x1) [draw, circle, fill=gray] {};
	\draw (4.5, 1) node (xg) [draw, circle, fill=gray] {};
	\draw (4.5, -1) node (vg) [draw, circle, fill=gray] {};

	\node at (-1.7,-1.3) {\small $\omega(y_{1})=1$};
	\node at (0.3,-1.3) {\small $\omega(y_{2})=1$};
	\node at (2.3,-1.3) {\small $\omega(y_{3})=1$};
	\node at (0.3,1.3) {\small $\omega(x_{2})=2$};
	\node at (-1.7,1.3) {\small $\omega(x_{1})=2$};
	\node at (2.3,1.3) {\small $\omega(x_{3})=1$};
    \node at (4.3,1.3) {\small $\omega(x_{4})=1$};
    \node at (4.3,-1.3) {\small $\omega(y_{4})=2$};
	
	\end{tikzpicture}

\caption{Cohen-Macaulay weighted oriented bipartite graph $\D$.}\label{graph1}
\end{center}
\end{figure}
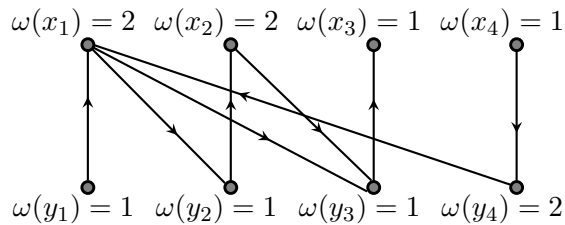

\bibliographystyle{plain}

\end{document}